\documentclass[preprint,5pt]{elsarticle}

\usepackage{amssymb}
\usepackage{exscale}
\usepackage{relsize}
\usepackage{amsmath,amsthm}
\usepackage{amsmath}
\usepackage{hyperref}
\usepackage{amssymb,latexsym}
\usepackage{mathrsfs}
\usepackage{enumerate}
\usepackage{indentfirst}
\usepackage{color}
\usepackage{tabularx}
\usepackage{booktabs}
\usepackage{makecell}
\usepackage{diagbox}
\usepackage{float}
\usepackage{algorithm}
\usepackage{multirow}
\usepackage{booktabs}
\usepackage{subcaption}   

\numberwithin{equation}{section}

\newtheorem{definition}{Definition}[section]
\newtheorem{example}{Example}[section]
\newtheorem{thm}{Theorem}[section]
\newtheorem{lem}{Lemma}[section]
\newdefinition{rmk}{Remark}[section]
\newtheorem{cor}{Corollary}[section]
\newcommand{\keywords}[1]{\par\noindent\textbf{Keywords:} #1}
\newproof{pf}{Proof}

\journal{Elsevier}

\begin{document}
	
	\begin{frontmatter}
		
		
		
		\title{ Data-driven accelerated multigrid method for Riesz space  fractional   diffusion equations                      }
				\author[]{Yang Hu$^{1}$ }
		
			\author[]{	Juan Zhang$^{*1,2}$\corref{cor1}}
		\ead{zhangjuan@xtu.edu.cn}

		\address{ \rm 1 School of Mathematics and Computational Science, Xiangtan University, Xiangtan, 411105,
			Hunan, People’s Republic of China}
		\address{ \rm 2 Hunan Key Laboratory for Computation and Simulation in Science and Engineering, Xiangtan
			University, Xiangtan, People’s Republic of China}

		\begin{abstract}
		\par
		\noindent   
The multigrid method is one of the most efficient iterative methods for solving large structured linear systems, and its efficiency stems from the complementarity  between the smoother and the coarse grid correction.  In this paper, we consider the block Toeplitz linear system obtained by discretizing time dependent Riesz space fractional diffusion equations, and employ the weighted Jacobi iterative method as a smoother. The weighted Jacobi smoother features a simple structure and low computational cost, yet its smoothing property is highly dependent on the relaxation parameter. We adopt Gaussian process regression to select the relaxation parameter to accelerate the convergence of the multigrid method. Numerical experiments confirm the theoretical analysis, which simultaneously verifies the effectiveness and superiority of the data-driven parameter selection strategy.

		\end{abstract}		
	    \end{frontmatter}

	\section{ Introduction}\label{Introduction}
Fractional differential equations are  powerful mathematical tools that can accurately describe certain complex phenomena, such as  pricing options \cite{Agrawal,Meer}, control problems  \cite{Du,Pou},  image processing \cite{IEEE,1D_Riesz}.      Consider a two-dimensional Riesz space fractional diffusion equation as follows: 
\begin{equation}
	\left\{
	\begin{aligned}
		&d(x, y, t) \frac{\partial u(x, y, t)}{\partial t} - \frac{\partial^{\gamma_{1}} u(x, y, t)}{\partial |x|^{\gamma_{1}}} - \frac{\partial^{\gamma_{2}} u(x, y, t)}{\partial |y|^{\gamma_{2}}} = f(x, y, t), && (x, y, t) \in \Omega \times (0, T), \\
		&u(x, y, t) = 0, && (x, y, t) \in \partial \Omega \times (0, T), \\
		&u(x, y, 0) = u_0(x, y), && (x, y) \in \bar{\Omega},
	\end{aligned}
	\right.\label{new1}
\end{equation}
where $\Omega =(x_{L},x_{R})\times (y_{L},y_{R})$, $\partial \Omega$ denotes the boundary of $\Omega$, $f(x,y,t)$ is the  source term, $d(x,y,t)\geq 0$ is the coefficient function. $u_{0}(x,y)$ is the initial condition, $\gamma_{1}, \gamma_{2}\in (1,2)$. Specifically, Riesz fractional derivatives are defined as follows:
\begin{align*}
	\frac{\partial^{\gamma_1} u(x, y, t)}{\partial |x|^{\gamma_1}} 
	&= \sigma_{\gamma_1} \left( {}_{x_{L} }D_x^{\gamma_1} + {}_x D_{x_{R}}^{\gamma_1} \right) u(x, y, t),  \\
	\frac{\partial^{\gamma_2} u(x, y, t)}{\partial |y|^{\gamma_2}} 
	&= \sigma_{\gamma_2} \left( {}_{y_{L}} D_y^{\gamma_2} + {}_y D_{y_{R}}^{\gamma_2} \right) u(x, y, t),
\end{align*}
	with $\sigma_{\gamma_{l}}=-\frac{1}{2\cos ({\pi \gamma_{l}}/{2})}>0$ for $l=1,2$.
Owing to the nonlocality of fractional differential operators, it is difficult to obtain the analytical solutions of fractional differential equations, and thus developing efficient numerical methods to acquire their numerical solutions has become a research focus in this field.

Discretizing equation (\ref{new1}) on a uniform grid using the shifted
Grünwald difference scheme in space and the backward Euler scheme in
time yields a sequence of Toeplitz-like linear systems. Due to the nonlocality of fractional differential operators, the coefficient matrices are generally dense, which renders the discrete linear systems computationally expensive for general purpose solvers.  For solving
such linear systems, \cite{BZZ} proposed a class of
diagonal--Toeplitz splitting (DTS) iteration methods, implemented by
alternately solving shifted diagonal and shifted Toeplitz subsystems.
Building upon the DTS framework, \cite{Lu} used a  
 natural $\tau$
approximation of the symmetric Toeplitz component and proposed a
sine-transform-based splitting preconditioner. More recently, \cite{She} decomposed the coefficient matrix by extracting a scalar matrix, a diagonal compensation matrix, and a block $\tau$ matrix. The sum of these matrices was then employed to construct a class of $\tau$-splitting iteration methods.


For certain structured matrices, the multigrid   method exhibits excellent solution efficiency\cite{Ruge,Lin,GLT2026}.   Under suitable conditions, this method enjoys mesh‑independent convergence.
 This advantageous property has drawn the attention of many researchers to investigate the application of  multigrid methods to fractional partial differential equations. 
 For instance,  \cite{PKH} employed the two-grid method to solve the dense Toeplitz linear system obtained after discretizing the one-dimensional constant-coefficient fractional diffusion equation, and utilized fast Toeplitz matrix--vector multiplication to reduce the computational overhead .  \cite{Serra_2016} proved the convergence of the multigrid method for two-dimensional space-fractional diffusion equations with constant diffusion coefficients.  Based on the generalized locally Toeplitz (GLT) sequences theory, they further extended this result to the case of uniformly bounded and positive variable
 diffusion coefficients.  Furthermore, numerous studies adopt the GLT theory to analyze the spectral properties of coefficient matrices and thereby design multigrid methods\cite{1D_Riesz,2D_Riesz,Serra_2025}. In addition,  \cite{Bolten_2022,Bolten_2025} extended the
 symbol-based multigrid analysis for scalar Toeplitz matrices
 to block-circulant and block-Toeplitz systems generated by
 matrix-valued symbols, and derived sufficient conditions for
 the two-grid and V-cycle methods .
 
 The multigrid method consists of a smoother and a coarse-grid correction  that complement each other. The coarse-grid correction is responsible for eliminating the low-frequency components of the error. For Toeplitz-like matrices, coarse-grid correction using bilinear interpolation as the transfer operator is effective and general \cite{Moghaderi}. Correspondingly, the smoother should rapidly attenuate the high-frequency components of the error that are difficult for the coarse grid to represent, and it is usually implemented by classical iterative methods. The weighted Jacobi method is widely adopted due to its simplicity and low computational cost per iteration.  With an appropriate relaxation parameter, this method exhibits favorable smoothing ability for high-frequency components of the error\cite{Tro,XJC}. Therefore, the proper choice of the relaxation parameter is the key to constructing an efficient smoother.
 
  Previously, research on the selection of relaxation parameters for the weighted Jacobi smoother has mainly focused on classical elliptic partial differential equations\cite{HG,HE}. \cite{Serra_2025}extended this study to dense Toeplitz-type linear systems arising from the discretization of tempered fractional diffusion equations. They used the spectral information provided by the  generating function of the coefficient matrix to expand the range of parameters that guarantee the smoothing property of the Jacobi smoother,  and gave estimation formulas for suitable relaxation parameters. However, while their estimation formulas yield feasible relaxation parameters, the resulting parameter may not always match the numerically optimal choice for finite‑dimensional linear systems.
 
 In recent years, the integration of machine learning with traditional numerical methods has demonstrated great potential. Machine learning techniques can identify nonlinear relationships between problem characteristics and numerical algorithms from data that are difficult to characterize explicitly, thereby providing new approaches to the design of complex numerical algorithms and the automatic selection of their parameters.  Moreover, once trained, machine learning models can typically make rapid predictions for new problems at a low computational cost. In the field of multigrid methods,  \cite{HRu,Nature,Ll} have applied machine learning to the design and optimization of smoothers.  In this work, we adopt Gaussian Process Regression (GPR) to learn and select the relaxation parameters of the weighted Jacobi smoother, so as to improve the convergence speed of the multigrid method. 
 	
 GPR is a Bayesian nonparametric learning method that places a Gaussian process prior on the unknown mapping from problem features to the desired relaxation parameter. Given a small set of training samples, the correlations among the samples are characterized by a covariance kernel, and conditioning the prior on the observed data yields a posterior predictive distribution for a new problem instance. Its mean can be used as the predicted relaxation parameter, while its variance provides a quantitative measure of prediction uncertainty. It can capture nonlinear parameter dependence without prescribing a fixed functional form, perform well with relatively small datasets, and, once trained, deliver predictions at a low online cost.


 The paper is organized as follows. In Section \ref{Sec2},we discretize the original equations and analyze the basic spectral properties of the coefficient matrix. In Section \ref{TGM}, with the aid of the Ruge–Stuben two grid framework, we design and analyze the convergence of the two grid method. In Section \ref{GPR}, we introduce the principles of GPR and its procedure for predicting parameters. In Section \ref{Sec5}, we present numerical experiments to demonstrate the effectiveness and superiority of the multigrid method and GPR for parameter prediction. In Section \ref{conclusion}, we summarize the conclusions and future work.

\section{ Discretization and matrix properties}\label{Sec2}
In this section, we will introduce the specific discrete schemes and analyze the properties of the coefficient matrices obtained from the discretization, which will help us design the multigrid method.
Below, we review the left- and right-sided Riemann–Liouville fractional derivatives:
\begin{align*}
	{}_{x_{L} }D_x^{\gamma_{1}} u(x,y,t) &= \frac{1}{\Gamma(2-\gamma_{1})} \frac{d^2}{dx^2} \int_{x_{L}}^{x} \frac{u(\xi,y,t)}{(x-\xi)^{\gamma_{1}-1}} d\xi, & 
{}_x D_{x_{R}}^{\gamma_{1}}u(x,y,t) &= \frac{1}{\Gamma(2-\gamma_{1})} \frac{d^2}{dx^2} \int_{x}^{x_{R}} \frac{u(\xi,y,t)}{(\xi-x)^{\gamma_{1}-1}} d\xi, \\
		{}_{y_{L} }D_y^{\gamma_{2}} u(x,y,t) &= \frac{1}{\Gamma(2-\gamma_{2})} \frac{d^2}{dy^2} \int_{y_{L}}^{y} \frac{u(x,\xi,t)}{(y-\xi)^{\gamma_{2}-1}} d\xi, & 
{}_y D_{y_{R}}^{\gamma_{2}}u(x,y,t)&= \frac{1}{\Gamma(2-\gamma_{2})} \frac{d^2}{dy^2} \int_{y}^{y_{R}} \frac{u(x,\xi,t)}{(\xi-y)^{\gamma_{{2}}-1}} d\xi,
\end{align*}
where $\Gamma(\cdot)$ is the Gamma function. Let $N,M $ be positive integers. Take the step sizes $\Delta x=\frac{x_{R}-x_{L}}{N+1}$, $\Delta y=\frac{y_{R}-y_{L}}{N+1}$, $\Delta t=\frac{T}{M}$ to discretize the $x$-direction, $y$-direction and time, yielding
\begin{align*}
	&\{x_{i}| x_{i}=x_{L}+i\Delta x, 0\leq i\leq N+1 \},\\
	&\{y_{j}| y_{j}=y_{L}+j\Delta y, 0\leq j\leq N+1 \},\\
	&\{t_{m}| t_{m}=m\Delta t, 0\leq m\leq M \}.
\end{align*}
Define $u^{(m)}_{i,j}=u(x_{i},y_{j},t_{m})$, $f^{(m)}_{i,j}=f(x_{i},y_{j},t_{m})$, $d^{(m)}_{i,j}=d(x_{i},y_{j},t_{m})$. Using the shifted $\text{Gr\"unwald}$ formula\cite{Meer0,Meer2}, we can obtain the approximation of the fractional derivative at the point $(x_{i},y_{j},t_{m})$:
\begin{align*}
	&{}_{x_{L}}D^{\gamma_{1}}_{x}(x_{i},y_{j},t_{m})= \frac{1}{{\Delta x}^{\gamma_{1}}} \sum_{k=0}^{i+1}g^{(\gamma_{1})}_{k}u^{(m)}_{i-k+1,j}+\mathcal{O}(\Delta x),\\
	&{}_{x}D^{\gamma_{1}}_{x_{R}}(x_{i},y_{j},t_{m})= \frac{1}{{\Delta x}^{\gamma_{1}}} \sum_{k=0}^{N-i+2}g^{(\gamma_{1})}_{k}u^{(m)}_{i+k-1,j}+\mathcal{O}(\Delta x),
\end{align*}
and
\begin{align*}
	&{}_{y_{L}}D^{\gamma_{2}}_{y}(x_{i},y_{j},t_{m})= \frac{1}{\Delta_{y}^{\gamma_{2}}} \sum_{k=0}^{j+1}g^{(\gamma_{2})}_{k}u^{(m)}_{i,j-k+1}+\mathcal{O}(\Delta y),\\
	&{}_{y}D^{\gamma_{2}}_{y_{R}}(x_{i},y_{j},t_{m})= \frac{1}{\Delta_{y}^{\gamma_{2}}} \sum_{k=0}^{N-j+2}g^{(\gamma_{2})}_{k}u^{(m)}_{i,j+k-1}+\mathcal{O}(\Delta y),
\end{align*}
where the coefficients $g^{(\gamma)}_{k}$ are given by the following lemma:

\begin{lem} \label{yl1}\rm  \cite{Lu}
For $ \gamma \in (1,2) $, the following recurrence relation holds:
\begin{align*}
g_k^{(\gamma)} = \left(1 - \frac{\gamma + 1}{k}\right) g_{k-1}^{(\gamma)}=\frac{(-1)^{k}}{k!}{\gamma(\gamma -1)\cdots(\gamma -k+1)}, \quad k = 1, 2, \cdots.
\end{align*}
Moreover, we have
\begin{align*}
g_0^{(\gamma)} = 1, \quad g_1^{(\gamma)} = -\gamma < 0, \quad 1 > g_2^{(\gamma)} > g_3^{(\gamma)} > \cdots > 0,
\end{align*}
and
\begin{align*}
	\sum_{k=0}^{\infty} g_k^{(\gamma)} = 0,
	\quad \text{or equivalently,} \quad
	\gamma = \sum_{\substack{k=0 \\ k \neq 1}}^{\infty} g_k^{(\gamma)}.
\end{align*}

\end{lem}
To facilitate the subsequent analysis, we introduce a matrix $T^{\gamma}_{N}\in \mathbb R^{N\times N}$ associated with $\gamma$:
\begin{equation}
	T^{\gamma}_{N} = -\left[
	\begin{array}{cccccc}
		2 g_{1}^{(\gamma)} & g_{0}^{(\gamma)}+g_{2}^{(\gamma)} & g_{3}^{(\gamma)} & \cdots & g_{N-1}^{(\gamma)} & g_{N}^{(\gamma)} \\
		g_{0}^{(\gamma)}+g_{2}^{(\gamma)} & 2 g_{1}^{(\gamma)} & g_{0}^{(\gamma)}+g_{2}^{(\gamma)} & \ddots & \cdots & g_{N-1}^{(\gamma)} \\
		g_{3}^{(\gamma)} & g_{0}^{(\gamma)}+g_{2}^{(\gamma)} & 2 g_{1}^{(\gamma)} & \ddots & \ddots & \vdots \\
		\vdots & \ddots & \ddots & \ddots & \ddots & g_{3}^{(\gamma)} \\
		g_{N-1}^{(\gamma)} & \cdots & \ddots & \ddots & 2 g_{1}^{(\gamma)} & g_{0}^{(\gamma)}+g_{2}^{(\gamma)} \\
		g_{N}^{(\gamma)} & g_{N-1}^{(\gamma)} & \cdots & g_{3}^{(\gamma)} & g_{0}^{(\gamma)}+g_{2}^{(\gamma)} & 2 g_{1}^{(\gamma)}
	\end{array}
	\right]. \label{jz1}
\end{equation}
According to Lemma \ref{yl1}, $T^{\gamma}_{N}$ is a matrix that is strictly diagonally dominant, symmetric, and has all positive diagonal entries. Therefore, $T^{\gamma}_{N}$ is a symmetric positive definite matrix. Let
\begin{align*}
	\textbf{u}^{(m)}&=[u^{(m)}_{1,1},u^{(m)}_{2,1} ,\dots,u^{(m)}_{N,1},u^{(m)}_{1,2},u^{(m)}_{2,2},\dots,u^{(m)}_{N,2},\dots,u^{(m)}_{1,N},u^{(m)}_{2,N},\dots,u^{(m)}_{N,N}]^{T},\\  
	\textbf{f}^{(m)}&=[f^{(m)}_{1,1},f^{(m)}_{2,1} ,\dots,f^{(m)}_{N,1},f^{(m)}_{1,2},f^{(m)}_{2,2},\dots,f^{(m)}_{N,2},\dots,f^{(m)}_{1,N},f^{(m)}_{2,N},\dots,f^{(m)}_{N,N}]^{T},\\
	{D}^{(m)}&={\text{diag}}(d^{(m)}_{1,1},d^{(m)}_{2,1} ,\dots,d^{(m)}_{N,1},d^{(m)}_{1,2},d^{(m)}_{2,2},\dots,d^{(m)}_{N,2},\dots,d^{(m)}_{1,N},d^{(m)}_{2,N},\dots,d^{(m)}_{N,N}).
\end{align*}
We discretize (\ref{new1}) in the spatial direction using the shifted Grunwald formula and in the temporal direction using the backward Euler scheme to obtain the implicit finite-difference scheme:
\begin{small}
	\begin{align*}
		d^{(m)}_{i,j}\frac{u_{i,j}^{(m)}-u_{i,j}^{(m-1)}}{\Delta t}-\frac{\sigma_{\gamma_{1}}}{{\Delta x}^{\gamma_{1}}}\left(    \sum_{k=0}^{i+1}g^{(\gamma_{1})}_{k}u^{(m)}_{i-k+1,j}+\sum_{k=0}^{N-i+2}g^{(\gamma_{1})}_{k}u^{(m)}_{i+k-1,j}         \right)-\frac{\sigma_{\gamma_{2}}}{{\Delta y}^{\gamma_{2}}}\left(    \sum_{k=0}^{j+1}g^{(\gamma_{2})}_{k}u^{(m)}_{i,j-k+1}+\sum_{k=0}^{N-j+2}g^{(\gamma_{2})}_{k}u^{(m)}_{i,j+k-1}         \right)=f^{(m)}_{i,j}.
	\end{align*}
\end{small}
Furthermore, we have
\begin{align*}
	{D}^{(m)}\frac{\textbf{u}^{(m)}-\textbf{u}^{(m-1)}}{\Delta t}+\left(\frac{\sigma_{\gamma_{1}}}{{\Delta x}^{\gamma_{1}}}I_{N}\otimes T_{\gamma_{1}}   +\frac{\sigma_{\gamma_{2}}}{{\Delta y}^{\gamma_{2}}}T_{\gamma_{2}}\otimes I_{N}              \right)\textbf{u}^{(m)}=\textbf{f}^{(m)}.
\end{align*}
Let $\eta_{1}=\sigma_{\gamma_{1}}\frac{\Delta t}{{\Delta x}^{\gamma_{1}}},\eta_{2}=\sigma_{\gamma_{2}}\frac{\Delta t}{{\Delta y}^{\gamma_{2}}}$. Consequently, we obtain
\begin{align}
	\left(	{D}^{(m)}+\eta_{1}I_{N}\otimes T^{\gamma_{1}}_{N}+\eta_{2} T^{\gamma_{2}}_{N}\otimes I_{N}\right)\textbf{u}^{(m)}=\Delta t \textbf {f}^{(m)}+{D}^{(m)}\textbf {u}^{(m-1)}. \label{eq2}
\end{align}
Dividing both sides of (\ref{eq2}) by $\eta_{1}$, we have
\begin{align*}
	\left(\frac{1}{\eta_{1}}	D^{(m)}+I_{N}\otimes T^{\gamma_{1}}_{N}+\frac{\eta_{2}}{\eta_{1}} T^{\gamma_{2}}_{N}\otimes I_{N}\right)\textbf{u}^{(m)}=\frac{1}{\eta_{1}}\Delta t \textbf {f}^{(m)}+\frac{1}{\eta_{1}}D^{(m)}\textbf {u}^{(m-1)}. 
\end{align*}
Define
\begin{align}
	&A_{N}^{(m)}=\frac{1}{\eta_{1}}	D^{(m)}+I_{N}\otimes T^{\gamma_{1}}_{N}+\frac{\eta_{2}}{\eta_{1}} T^{\gamma_{2}}_{N}\otimes I_{N},\notag\\
	&\textbf{b}^{(m)}=\frac{1}{\eta_{1}}\Delta t \textbf {f}^{(m)}+\frac{1}{\eta_{1}}D^{(m)}\textbf {u}^{(m-1)}.\label{ANm}
\end{align}
Then (\ref{eq2}) is equivalent to
\begin{align}
	A_{N}^{(m)}\textbf{u}^{(m)}=\mathbf{b}^{(m)},\qquad 1\leq m \leq M. \label{eq5}
\end{align}
After fixing the partition number $N$, discretizing \eqref{new1} yields a set of matrices $\{A^{(m)}_{N}\}$. When the  coefficient $d(x,y,t)=0$, $D^{(m)}$ degenerates to the zero matrix, and in this case, $A^{(m)}_{N}$ is independent of the time level $m$ and is denoted by $A_{N}$. The matrix $A_{N}$ is symmetric positive definite  and strictly diagonally dominant, and it is a Block-Toeplitz-Toeplitz-Block (BTTB) matrix.

 BTTB matrices belong to the class of multilevel Toeplitz matrices, and their spectral information is closely related to the generating function of the corresponding BTTB matrix sequence. Let us now recall the definitions of Toeplitz matrix sequences and BTTB matrix sequences.


\begin{definition}\rm \cite{Moghaderi}   Let $f(\theta)\in L^{1}([-\pi,\pi])$ and let $\{f_{k}\}$ be the Fourier coefficients defined as
	\begin{align*}
		f_{{k}} := \frac{1}{2\pi} \int_{[-\pi,\pi]} f({\theta}) {e}^{-{\mathbf{i}} k\theta} \mathrm{d}{\theta},
		\end{align*}
where $\mathbf{i}^{2}=-1$. Then the Toeplitz matrix of order $n$ associated with $f(\theta)$ is
\begin{align*}
	T_{n}(f)=[f_{i-j}]_{i,j=1}^{n}.
\end{align*}	
The function $f(\theta)$ is called the generating  function of the matrix sequence $\{T_n(f)\}_{n\in \mathbb N}$.

	\end{definition}

\begin{definition}\rm  \cite{Moghaderi} Let $f(\theta_{1},\theta_{2})\in L^{1}([-\pi,\pi]^{2})$ and let $\{f_{[k_{1},k_{2}]}\}$ be the Fourier coefficients defined as
	\begin{align*}
		f_{[k_{1},k_{2}]} = \frac{1}{4\pi^2} \int_{-\pi}^{\pi} \int_{-\pi}^{\pi} f(\theta_{1},\theta_{2}) e^{-\mathbf{i}(k_{1}\theta_{1} + k_{2}\theta_{2})} \, \mathrm{d}\theta_{1} \, \mathrm{d}\theta_{2}. 
	\end{align*}
	Then the Block-Toeplitz-Toeplitz-Blocks matrix of order $n=n_{1}n_{2}$ associated with $f(\theta_{1},\theta_{2})$ is
	\begin{align*}
T_{n}^{(2)}(f)=\left[\left[f_{[i_{1}-j_{1},\,i_{2}-j_{2}]}\right]_{i_{2},j_{2}=1}^{n_{2}}\right]_{i_{1},j_{1}=1}^{n_{1}}.
\end{align*}
Accordingly, $f(\theta_{1},\theta_{2})$ is called the generating  function of matrix sequence $\{T^{(2)}_n(f)\}_{n\in \mathbb N}$.  Moreover, $T^{(2)}_{n}(f)$  also admits the following representation:
\begin{align}
	T^{(2)}_{n}(f)=\underset{|k_{1}|< n_{1}}\sum\underset{|k_{2}|< n_{2}}\sum f_{[k_{1},k_{2}]}J^{[k_{1}]}_{n_{1}}\otimes J^{[k_{2}]}_{n_{2}}, \label{bttbdy}
\end{align}
	where $J^{[k]}_{n}$ is the matrix of order $n$ whose $(i,j)$ entry equals $1$ if $i-j=k$ and zero otherwise.
\end{definition}

In order to obtain the generating function of  $\{A_N\}_{N\in\mathbb N}$, we first give the following lemma.
\begin{lem}\rm \cite{PKH_2016,Serra_2016} Let $\gamma \in (1,2)$.  The  generating function of $\{T^{\gamma}_{N}\}_{N\in \mathbb N}$ in (\ref{jz1}) is real-valued and even, and its  specific form is as follows:
	\begin{align}
		f_{\gamma}(\theta)=-2^{\gamma +1}(\sin \frac{|\theta|}{2})^{\gamma}\cos (|\theta|(1-\frac{\gamma}{2})+\frac{\gamma \pi }{2}).\label{fgamma}
	\end{align}
\end{lem}

	\begin{proof}

		It is noted that the  generating function of $\{T^{\gamma}_{N}\}_{N\in \mathbb N}$ is
		\begin{align*}
			f_{\gamma}(\theta)=\sum_{k=-\infty}^{\infty}f_{k}e^{\mathbf{i}k\theta},
		\end{align*}
		where the coefficients are given by
		\begin{align*}
		f_{0}=-2g^{(\gamma)}_{1},\qquad
		f_{1}=f_{-1}=-g^{(\gamma)}_{0}-g^{(\gamma)}_{2},\qquad
		f_{k}=f_{-k}=-g^{(\gamma)}_{k+1},\quad k\geq 2.
		\end{align*}	
Furthermore, we obtain that
		\begin{align*}
			\sum_{k=-\infty}^{\infty}f_{k}e^{\mathbf{i}k\theta}=-\sum_{k=-1}^{\infty} g^{(\gamma)}_{k+1}e^{\mathbf{i}k\theta}-\sum_{k=-1}^{\infty} g^{(\gamma)}_{k+1}e^{-\mathbf{i}k\theta}.
		\end{align*}
 By virtue of the results in \cite{Serra_2016}, we have
	\begin{align*}
	-\sum_{k=-1}^{\infty} g^{(\gamma)}_{k+1}e^{\mathbf{i}k\theta}=-e^{-\mathbf{i}\theta}(1-e^{\mathbf{i}\theta})^{\gamma},	
	\end{align*}	
	and
	\begin{align*}
		-\sum_{k=-1}^{\infty} g^{(\gamma)}_{k+1}e^{-\mathbf{i}k\theta}=-e^{\mathbf{i}\theta}(1-e^{-\mathbf{i}\theta})^{\gamma}.	
	\end{align*}	
	Owing to
		\[
		\begin{minipage}{0.48\textwidth}
			\[
			1-e^{\mathbf{i}\theta}
			=
			\begin{cases}
				2\left|\sin\dfrac{\theta}{2}\right|\,
				e^{\frac{\mathbf{i}}{2}(\theta-\pi)},
				& 0\leq \theta\leq \pi, \\[0.8em]
				2\left|\sin\dfrac{\theta}{2}\right|\,
				e^{\frac{\mathbf{i}}{2}(\theta+\pi)},
				& -\pi\leq\theta<0,
			\end{cases}
			\]
		\end{minipage}
		\hfill {\text{and}}
		\begin{minipage}{0.48\textwidth}
			\[
			1-e^{-\mathbf{i}\theta}
			=
			\begin{cases}
				2\left|\sin\dfrac{\theta}{2}\right|\,
				e^{-\frac{\mathbf{i}}{2}(\theta-\pi)},
				& 0\leq \theta\leq\pi, \\[0.8em]
				2\left|\sin\dfrac{\theta}{2}\right|\,
				e^{-\frac{\mathbf{i}}{2}(\theta+\pi)},
				& -\pi\leq\theta<0.
			\end{cases}
			\]
		\end{minipage}
		\]
	When $0\leq \theta \leq\pi$,
	\begin{align*}
		f_{\gamma}(\theta)&=-e^{-\mathbf{i}\theta}(1-e^{\mathbf{i}\theta})^{\gamma}-e^{\mathbf{i}\theta}(1-e^{-\mathbf{i}\theta})^{\gamma}\\
		&=-e^{-\mathbf{i}\theta}2^{\gamma} \left|\sin^{\gamma}\dfrac{\theta}{2}\right|	e^{\frac{\mathbf{i}\gamma}{2}(\theta-\pi)}-e^{\mathbf{i}\theta}2^{\gamma} \left|\sin^{\gamma}\dfrac{\theta}{2}\right|	e^{\frac{-\mathbf{i}\gamma}{2}(\theta-\pi)}\\
		&=-2^{\gamma} \left|\sin^{\gamma}\dfrac{\theta}{2}\right|\left(  e^{\mathbf{i}\left(\theta(\frac{\gamma}{2}-1)-\frac{\pi \gamma}{2} \right)}+  e^{\mathbf{i}\left(-\theta(\frac{\gamma}{2}-1)+\frac{\pi \gamma}{2} \right)}                 \right)\\
		&=-2^{\gamma+1} \left|\sin^{\gamma}\dfrac{\theta}{2}\right|\cos\left(\theta(1-\frac{\gamma}{2})+\frac{\pi \gamma}{2}  \right).
	\end{align*}			
	When $-\pi< \theta \leq0$,		
		\begin{align*}
			f_{\gamma}(\theta)&=-2^{\gamma} \left|\sin^{\gamma}\dfrac{\theta}{2}\right|\left(   e^{\mathbf{i}\left(\theta(\frac{\gamma}{2}-1)+\frac{\pi \gamma}{2} \right)}+  e^{\mathbf{i}\left(-\theta(\frac{\gamma}{2}-1)-\frac{\pi \gamma}{2} \right)}       \right)\\
			&=-2^{\gamma+1} \left|\sin^{\gamma}\dfrac{\theta}{2}\right|\cos\left(-\theta(1-\frac{\gamma}{2})+\frac{\pi \gamma}{2}  \right).
		\end{align*}
		Combining the two gives (\ref{fgamma}). 
		
		\noindent The Proof is completed.
	\end{proof}

With the above preparations, we obtain the generating function of the BTTB matrix sequence  $\{A_N\}_{N\in\mathbb N}$.
\begin{thm}\rm
	Let $\gamma_{1},\gamma_{2}\in (1,2)$. Then $A_{N}$ is the BTTB matrix given by $A_{N}=T^{(2)}_{n}(f_{\gamma_{1},\gamma_{2}})$ with
	\begin{align}
		f_{\gamma_{1},\gamma_{2}}(\theta_{1},\theta_{2})=f_{\gamma_{1}}(\theta_{1})+\frac{\eta_{2}}{\eta_{1}}f_{\gamma_{2}}(\theta_{2}),\label{thm3.1}
	\end{align}
	where $\eta_{1}=\sigma_{\gamma_{1}}\frac{\Delta t}{\Delta x^{\gamma_{1}}},\quad \eta_{2}=\sigma_{\gamma_{2}}\frac{\Delta t}{\Delta y^{\gamma_{2}}}$, and $f_{\gamma}(\theta)$ is defined in (\ref{fgamma}).
\end{thm}
\begin{proof}
	By the definition of $A_{N}$, we have
	\begin{align*}
	A_{N}=I_{N}\otimes T_{\gamma_{1}}+\frac{\eta_{2}}{\eta_{1}} T_{\gamma_{2}}\otimes I_{N},
	\end{align*}
	where
	\begin{align*}
		&I_{N}\otimes T_{\gamma_{1}}= T_{N}(1)\otimes T_{N}(f_{\gamma_{1}})=T^{(2)}_{n}(\widetilde f_{\gamma_{1}}),\\
		&T_{\gamma_{2}}\otimes I_{N}= T_{N}(f_{\gamma_{2}})\otimes T_{N}(1)=T^{(2)}_{n}(\widetilde f_{\gamma_{2}}).
	\end{align*}
		Moreover, the Fourier coefficients of $\widetilde f_{\gamma_{1}}$ and
	$\widetilde f_{\gamma_{2}}$ are given by
	\begin{align*}
		\left(\widetilde{f}_{\gamma_{1}}\right)_{[k_{1},k_{2}]}
		&=
		\begin{cases}
			(f_{\gamma_{1}})_{k_{1}},
			& k_{2}=0, \\[0.8em]
			0,
			& k_{2}\neq 0,
		\end{cases}
		\quad \text{and} \quad
		\left(\widetilde{f}_{\gamma_{2}}\right)_{[k_{1},k_{2}]}
		=
		\begin{cases}
			(f_{\gamma_{2}})_{k_{2}},
			& k_{1}=0, \\[0.8em]
			0,
			& k_{1}\neq 0.
		\end{cases}
	\end{align*}
	Hence, $\widetilde{f}_{\gamma_{1}}(\theta_{1},\theta_{2})=f_{\gamma_{1}}(\theta_{1})$  and $\widetilde{f}_{\gamma_{2}}(\theta_{1},\theta_{2})=f_{\gamma_{2}}(\theta_{2})$. 	Therefore, the  generating function of $\{A_{N}\}_{N\in \mathbb N}$ is
		\begin{align*}
		f_{\gamma_{1},\gamma_{2}}(\theta_{1},\theta_{2})=f_{\gamma_{1}}(\theta_{1})+\frac{\eta_{2}}{\eta_{1}}f_{\gamma_{2}}(\theta_{2}).
	\end{align*}
	The proof is completed.	
\end{proof}

\begin{rmk}\rm If a bivariate function \(f(\theta_1,\theta_2)\) satisfies \(f(\theta_1,\theta_2)=f(|\theta_1|,|\theta_2|)\), then \(f(\theta_1,\theta_2)\) is called an even function. It is easy to verify that \(f_{\gamma_1,\gamma_2}(\theta_{1},\theta_{2})\) is a bivariate continuous real-valued even function.
	\end{rmk}


For a real-valued continuous even function $f(\theta_1,\theta_2)$, the associated BTTB matrix $T_n^{(2)}(f)$ is real symmetric. To characterize the spectrum of the matrices generated by $f_{\gamma_1,\gamma_2}(\theta_{1},\theta_{2})$, we recall the following standard result, which bounds the eigenvalues in terms of the minimum and maximum of the generating function:

\begin{lem}\rm \cite{SHW_2001} \label{lm3.2} 
	Let $f(\theta_{1},\theta_{2})\in C^{1}([-\pi,\pi]^{2})$ be a real-valued  even function , with its maximum and minimum denoted by $f_{\max}$ and $f_{\min}$, respectively. Then $T^{(2)}_{n}(f)$ is a real symmetric matrix, and its eigenvalues are bounded by the range of $f(\theta_{1},\theta_{2})$ in the sense that
	\begin{equation}
		f_{\min} \leq \lambda_{\min}\left(T^{(2)}_{n}(f)\right) \le \lambda_{\max}\left(T^{(2)}_{n}(f)\right) \leq f_{\max}.
	\end{equation}
\end{lem}

\section{Two Grid method}\label{TGM}
The Galerkin multigrid method belongs to the class of classical multigrid methods and is primarily applicable to symmetric positive definite linear systems $Au=b$. Its simplest form is the two grid method (TGM), which consists of a smoother and a coarse-grid correction. Multi-level schemes such as the V-cycle are realized by recursively cycling the TGM. The specific procedure of the TGM is shown in Table \ref{Table1}.
 In this paper, we choose the
weighted Jacobi  as the smoother, whose update formula is
\begin{align*}
	u^{(j+1)}
	= Gu^{(j)} + \omega \operatorname{diag}(A)^{-1} b,
\end{align*}
where \(u^{(j)}\) is the iteration vector at the \(j\)-th step,  and
\begin{align}
G = I_{n}-\omega\operatorname{diag}(A)^{-1}A.
\end{align}
 The coarse grid correction is obtained by combining the restriction operator $I^{H}_{h}$ and the prolongation operator $I^{h}_{H}$, where $I^{H}_{h}$ is a rectangular matrix with full row rank.
Specifically, we take
\begin{align}
I_h^H=P\otimes P, \label{tssz}
\end{align}
where
\[
P=
\begin{bmatrix}
	1 & 2 & 1 &        &        \\
	& 1 & 2 & 1      &        \\
	&   & \ddots & \ddots & \ddots \\
	&   &        & 1      & 2 & 1
\end{bmatrix}\in {\mathbb R^{ [\frac{N}{2}]\times N      }}.
\]
The prolongation operator is defined by
\begin{align}
I_H^h=\left(I_h^H\right)^T=P^{T}\otimes P^{T}, \label{czsz}
\end{align}
which defines a Galerkin coarse matrix $A_{H}=I^{H}_{h}AI^{h}_{H}$. 
\begin{table}[htbp] 
	\centering
	\renewcommand{\arraystretch}{1.3}
	\caption {Two Grid method} \label{Table1}
	\begin{tabular}{@{}>{\raggedright\arraybackslash}p{0.9\textwidth}@{}}
		\toprule

		1) \textbf{Pre-smoothing:} Apply $j_1$ iterations of $G$ to the system $Au=b$
		 to obtain $u^{(j_1)}$\\
		2) \textbf{Compute the residual:} $r=b-Au^{(j_{1})}$\\
		 3) \textbf{Restriction:} $b_{H}=I^{H}_{h}r$\\
		4)\textbf{Coarse grid equation solving:} Solve the system $A_{H}e = b_{H}$ exactly to obtain $\hat{e}$ \\
		5) \textbf{Update solution}  $\hat{u}=u^{(j_{1})}+I^{h}_{H}\hat {e}$\\
		6) \textbf{Post-smoothing:} Starting from $\hat{u}$, perform $j_2$ iterations using $G$\\
		\bottomrule
	\end{tabular}
\end{table}

\noindent Take $j_{1}=0$, $j_{2}=1$.
In this setting, the iteration matrix of the TGM is given as follows:
\begin{align}
	E_{TG}=G\left(I_{n}- I^{h}_{H}(I^{H}_{h}AI^{h}_{H})^{-1}I^{H}_{h}A    \right).
\end{align}
$E_{TG}$ is also referred to as the two grid error propagation operator. To characterize its convergence, we introduce the following inner product and the corresponding induced energy norm.

\begin{definition}\rm
	Let $\langle u, v \rangle$ be the standard Euclidean inner product, and let $A=(a_{ij})\in \mathbb R^{n\times n}$ be a symmetric  positive definite matrix. Define
	\begin{align}
		\langle u,v\rangle_{0,A}=\langle {\text {diag}(A)}u,v\rangle, \quad	\langle u,v\rangle_{1,A}=\langle Au,v\rangle, \quad	\langle u,v\rangle_{2,A}=\langle \left(\text{diag}(A)\right)^{-1}Au,Av\rangle. 	
	\end{align}
	The corresponding energy norms are
	\begin{align}
		||u||_{0,A}=\sqrt{\langle u,u \rangle_{0,A}},\quad 	||u||_{1,A}=\sqrt{\langle u,u \rangle_{1,A}},\quad 	||u||_{2,A}=\sqrt{\langle u,u \rangle_{2,A}}.
	\end{align}
\end{definition}

For the  TGM, Ruge and Stüben established the following analytical framework, which reduces the overall convergence analysis of the method to the separate examination and characterization of the properties of the smoother and prolongation operators.

\begin{lem}\rm \cite{Ruge}  \label{zyyl}
Let $A=(a_{ij}) \in \mathbb{R}^{n \times n}$ be a symmetric positive definite matrix, and let $G$ be the smoother for the TGM. Suppose the prolongation operator $I_{H}^{h}\in \mathbb R^{n\times n_{c}}$ $(n_{c}<n)$ is a column-full-rank matrix. If\\
	(1): There exists a constant $\alpha>0$ such that
	\begin{align}
		\|Ge^{h}\|^{2}_{1,A}\leq \|e^{h}\|_{1,A}^{2}-\alpha \|e^{h}\|_{2,A}^{2}, \forall e^{h}\in \mathbb R^{n}. \quad{\text{(Smoothing property)}} 
		\label{mgsz}
	\end{align}
	\\
	(2): There exists a constant $\beta>0$ such that
	\begin{align}
		\min_{e^{H}\in \mathbb{R}^{n_{c}}} \|e^{h} - I^{h}_{H} e^{H}\|^{2}_{0,A} \leq \beta \|e^{h}\|^{2}_{1,A}, \quad \forall e^{h}\in \mathbb{R}^{n}. \quad{\text{(Approximation property)}}\label{ytsz} 
	\end{align}
	Then $\beta\geq \alpha$ and spectral radius of $E_{TG}$ satisfies
	\begin{equation*}
		\rho(E_{TG})\leq 	\|E_{TG}\|_{1,A}\leq {\sqrt{1-\frac{\alpha}{\beta}}}<1. 
	\end{equation*}
\end{lem}

Ruge and St\"uben also rigorously discussed the range of the relaxation parameter $\omega$ when the smoother $G$ is taken as the weighted Jacobi iteration method, in order to guarantee the smoothing property, namely the following lemma:

\begin{lem}  \rm \cite{Ruge}
Let $A=(a_{ij})\in\mathbb{R}^{n\times n}$ be a symmetric positive definite matrix, and let
\[
G=I_n-\omega\,\operatorname{diag}(A)^{-1}A.
\]
If the relaxation parameter $\omega$ satisfies
\[
\omega\in\left(0,\frac{2}
{\rho\!\left(\operatorname{diag}(A)^{-1}A\right)}\right),
\]
then there exists a constant $\alpha>0$ such that inequality~(\ref{mgsz}) in Lemma~\ref{zyyl} holds.
\end{lem}

We first establish a TGM analysis for $A_N$, 
verifying that it satisfies both the smoothing property and the approximation property. 
Building upon this, by utilizing the perturbation property of positive diagonal matrices, 
we further derive the TGM analysis for $A_N^{(m)}$. 

\begin{thm} \rm
	 Let $A_{N}=T^{(2)}_{n}\left(f_{\gamma_{1},\gamma_{2}}(\theta_{1},\theta_{2})\right)$, where $f_{\gamma_{1},\gamma_{2}}(\theta_{1},\theta_{2})$ is given by (\ref{thm3.1}). Let $G=I_{n}-\omega \left(\mathrm{diag}(A_{N})\right )^{-1}A_{N}$. If $\omega$ satisfies $0<\omega<\frac{2a_{0}}{\|f_{\gamma_{1},\gamma_{2}}\|_{\infty}}$, where $a_{0}=-2g^{(\gamma_{1})}_{1}-\frac{2\eta_{2}}{\eta_{1}}g^{(\gamma_{2})}_{1}$, then there exists a  constant $\alpha$ satisfying (\ref{mgsz}) in Lemma \ref{zyyl}.
	\end{thm}
\begin{proof}	
	Since all diagonal entries of $A_{N}$ are $a_{0}$, we have
	\begin{align*}
		\frac{2}{\lambda_{\max}\!\left(\operatorname{diag}(A_{N})^{-1}A_{N}\right)}
		=\frac{2a_{0}}{\lambda_{\max}\!\left(A_{N}\right)}.
	\end{align*}
	By Lemma \ref{lm3.2}, we obtain
	\[
	\frac{2a_{0}}{\|f_{\gamma_{1},\gamma_{2}}\|_{\infty}}
	\leq \frac{2a_{0}}{\lambda_{\max}\!\left(A_{N}\right)}.
	\]
	Thus, from Lemma 3.2, a sufficient condition for $\omega$ is
	\[
	\omega\in\left(0,\frac{2a_{0}}{\|f_{\gamma_{1},\gamma_{2}}\|_{\infty}}\right).
	\]
	\noindent The Proof is completed.
\end{proof}

The approximation properties of prolongation operators were studied in
\cite{SHW_2001}. In particular, sufficient conditions were established
under which a BTTB matrix generated by a nonnegative even function ensures
that the $I^{h}_{H}$ defined by (\ref{tssz}) satisfies
(\ref{ytsz}). More precisely, the following lemma holds.

\begin{lem}\rm   \cite{SHW_2001} If the nonnegative even function $f(\theta_{1},\theta_{2})\in C^{1}[-\pi,\pi]^{2}$ satisfies:
	\begin{align*}
		\underset{(\theta_{1},\theta_{2})\in [-\pi,\pi]^{2} }{\rm min}\frac{f(\theta_{1},\theta_{2})}{2-\cos\theta_{1}-\cos\theta_{2}}>0,
	\end{align*}
	then the BTTB matrix $T^{(2)}_{n}(f)$ generated by $f$ satisfies the following
	property: there exists a constant $\beta$ such that (\ref{ytsz}) holds,
	where the prolongation operator $I^{h}_{H}$ is defined by (\ref{czsz}). 
	
\end{lem}

\begin{thm}\rm  Let $A_{N}=T^{(2)}_{n}\left(f_{\gamma_{1},\gamma_{2}}(\theta_{1},\theta_{2})\right)$, where $f_{\gamma_{1},\gamma_{2}}(\theta_{1},\theta_{2})$ is given by (\ref{thm3.1}). Let $I^{h}_{H}$ be defined as in  \eqref{czsz}.  Then there exists a  constant $\beta$ satisfying (\ref{ytsz}) in Lemma \ref{zyyl}.
	\end{thm}

\begin{proof}
	Set
	\[
	\eta=\frac{\eta_{2}}{\eta_{1}}>0.
	\]
	Define
	\[
	\phi_{\gamma}(\theta)
	=
	|\theta|\left(1-\frac{\gamma}{2}\right)
	+\frac{\gamma\pi}{2}.
	\]
	Since $|\theta|\in[0,\pi]$ and $\gamma\in(1,2)$, we have
\begin{align*}
	-\cos\left(\phi_{\gamma}(\theta)\right)\geq -\cos(\frac{\gamma \pi}{2})
\end{align*}
	and
\begin{align*}
	\left(\sin \frac{|\theta|}{2}\right)^{\gamma}\geq \sin^{2}\frac{|\theta|}{2}=\frac{(1-\cos\theta)}{2}.
\end{align*}
	Then 
	\begin{align*}
		f_{\gamma}(\theta)&=-2^{\gamma+1}\left( \sin\frac{|\theta|}{2} \right)^{\gamma}\cos\left(\phi_{\gamma}(\theta)\right)\\
		&\geq -2^{\gamma }\cos\left(\frac{\gamma \pi}{2}\right)(1-\cos \theta)
	\end{align*}
	Consequently, for every	$(\theta_{1},\theta_{2})\neq(0,0)$,
	\begin{align*}
		f_{\gamma_{1},\gamma_{2}}(\theta_{1},\theta_{2})
		&=
		f_{\gamma_{1}}(\theta_{1})
		+\eta f_{\gamma_{2}}(\theta_{2}) \\
		&\geq
-2^{\gamma_{1}}\left(\cos\frac{\gamma_{1}\pi }{2}\right)(1-\cos\theta_{1})-\eta2^{\gamma_{2}}\left(\cos\frac{\gamma_{2}\pi }{2}\right)(1-\cos\theta_{2})\\
		&\geq
		\min\left\{-2^{\gamma_{1}}\left(\cos\frac{\gamma_{1}\pi }{2}\right),
	-\eta2^{\gamma_{2}}\left(\cos\frac{\gamma_{2}\pi }{2}\right)
		\right\}
		\left(
		2-\cos\theta_{1}-\cos\theta_{2}
		\right).
	\end{align*}
	Hence,
	\begin{align*}
		\frac{
			f_{\gamma_{1},\gamma_{2}}(\theta_{1},\theta_{2})
		}{
			2-\cos\theta_{1}-\cos\theta_{2}
		}
		&\geq
		\min\left\{-2^{\gamma_{1}}\left(\cos\frac{\gamma_{1}\pi }{2}\right),
	-\eta2^{\gamma_{2}}\left(\cos\frac{\gamma_{2}\pi }{2}\right)
	\right\}	>0.
	\end{align*}
	Therefore,
	\[
	\min_{\substack{
			(\theta_{1},\theta_{2})\in[-\pi,\pi]^{2}\\
			(\theta_{1},\theta_{2})\neq(0,0)
	}}
	\frac{
		f_{\gamma_{1},\gamma_{2}}(\theta_{1},\theta_{2})
	}{
		2-\cos\theta_{1}-\cos\theta_{2}
	}
	>0.
	\]
	Owing to
	\begin{align*}
\lim\limits_{\theta \to 0}
\frac{f_{\gamma}(\theta)}
{-2\cos\left(\frac{\gamma \pi}{2}\right)|\theta|^{\gamma}}
=1
\end{align*}	
	and
	\begin{align*}
		\lim\limits_{\left(\theta_{1},\theta_{2}\right) \to (0,0)}\frac{\theta^{2}_{1}+\theta^{2}_{2}}{2-\cos \theta_{1}-\cos \theta_{2}}=2.
	\end{align*}
	Then
	\begin{align*}
			\lim\limits_{\left(\theta_{1},\theta_{2}\right) \to (0,0)}	\frac{
			f_{\gamma_{1},\gamma_{2}}(\theta_{1},\theta_{2})
		}{
			2-\cos\theta_{1}-\cos\theta_{2}
		}=+\infty.	
	\end{align*}
	Thus, the quotient cannot attain its infimum near the origin. Since it
	is continuous away from the origin, its infimum is attained on a
	compact subset of
	$[-\pi,\pi]^{2}\setminus\{(0,0)\}$.
	Therefore, the infimum may equivalently be written as a minimum.
	
	\noindent The Proof is completed.
\end{proof}

Several properties obtained by applying positive diagonal matrix perturbations are presented below. If $A$ and $B$ are symmetric positive definite matrices, the notation $A\le B$ means that $B-A$ is a nonnegative definite matrix.

\begin{lem}\rm \label{yl4.4}
	Let $A=(a_{ij}) \in \mathbb{R}^{n\times n}$ be a symmetric positive definite matrix , and let $B=A+D$, where $D$ is a nonnegative diagonal matrix. Then
	\begin{align}
		\rho(\operatorname{diag}(B)^{-1}B)
\leq 	\rho(\operatorname{diag}(A)^{-1}A)  . \label{ine0}
	\end{align}
\end{lem}

\begin{proof}

We denote $D_0 = \operatorname{diag}(A)$. Let
\begin{align*}
	l = \rho(D_0^{-1}A) = \lambda_{\max}\bigl(D_0^{-\frac{1}{2}} A D_0^{-\frac{1}{2}}\bigr).
\end{align*}
Since $\operatorname{tr}(D_0^{-\frac{1}{2}} A D_0^{-\frac{1}{2}})=n$, we have $l \geq 1$ and
\begin{align*}
	D_0^{-\frac{1}{2}} A D_0^{-\frac{1}{2}}\leq  l I_n,
\end{align*}
i.e.,
\begin{align}\label{ine1}
	A\leq lD_0.
\end{align}
Since $D\geq 0$ and $l\geq 1$, we have
\begin{align}
	D\leq lD. \label{ine2}
\end{align}
Adding \eqref{ine1} and \eqref{ine2}, we obtain
\begin{align}
	A+D\leq l(D_0+D).
\end{align}
Applying the congruence transformation induced by $(D+D_0)^{-\frac{1}{2}}$ to both sides of the above inequality yields
\begin{align*}
	(D+D_0)^{-\frac{1}{2}}(A+D)(D+D_0)^{-\frac{1}{2}}\leq lI_n.
\end{align*}
Taking the maximum eigenvalue on both sides of the above matrix inequality yields \eqref{ine0}.\\
The proof is complete.
\end{proof}

\begin{lem}\rm \label{xyl2}
	Let $A=(a_{ij})\in\mathbb R^{n\times n}$ be a symmetric positive definite matrix, and there exists a constant $\beta>0$ such that for any $e^h\in\mathbb R^n$,
	\begin{align*}
		\min_{e^H\in\mathbb R^{n_c}} \|e^h - I_H^h e^H\|_{0,A}^2 \le \beta \|e^h\|_{1,A}^2 .
	\end{align*}
	Define $B=A+D$, where $D=(d_{ij})\in\mathbb R^{n\times n}$ is a nonnegative  diagonal matrix. Then for any $e^h\in\mathbb R^n$, we have
	\[
	\min_{e^H\in\mathbb R^{n_c}} \|e^h - I_H^h e^H\|_{0,B}^2 \le \beta\left(1+\frac{ d_{\max}}{a_{\min}}\right)\|e^h\|_{1,B}^2,
	\]
	where $ d_{\max} = \max\{d_{11},\dots,d_{nn}\}$ and $a_{\min} = \min\{a_{11},\dots,a_{nn}\}$.
\end{lem}
\begin{proof}
	
	For any $e^{h}\in \mathbb R^{n}$ and $e^{H}\in \mathbb R^{n_{c}}$, we have
	\begin{align*}
		\|e^{h}-I^{h}_{H}e^{H}\|^{2}_{0,B}
		&= \|e^{h}-I^{h}_{H}e^{H}\|^{2}_{0,A} + \|e^{h}-I^{h}_{H}e^{H}\|^{2}_{0,D} \\
		&\leq \left(1+\frac{ d_{\max}}{a_{\min}}\right) \|e^{h}-I^{h}_{H}e^{H}\|^{2}_{0,A}.
	\end{align*}
	Now, suppose that for a fixed $e^h \in \mathbb R^n$, $e^H(h)$ is chosen such that
	\begin{align*}
	\min_{e^H \in \mathbb R^{n_c}} \|e^h - I_H^h e^H\|_{0,A}^2
	= \|e^h - I_H^h e^H(h)\|_{0,A}^2 .
\end{align*}
	Then, for any $e^{h}\in \mathbb R^{n}$
	\begin{align*}
			&\min_{e^H \in \mathbb R^{n_c}} \|e^h - I_H^h e^H\|_{0,B}^2\\
			&\leq  \|e^h - I_H^h e^H(h)\|_{0,B}^2\\
			&\leq \left(1+\frac{ d_{\max}}{a_{\min}}\right)\|e^{h}-I^{h}_{H}e^{H}(h)\|^{2}_{0,A}\\
		&	\leq  \left(1+\frac{ d_{\max}}{a_{\min}}\right)\beta ||e^{h}||^{2}_{1,A}\\
			&\leq \left(1+\frac{ d_{\max}}{a_{\min}}\right)\beta ||e^{h}||^{2}_{1,B}.
	\end{align*}
The proof is completed.	
\end{proof}

The convergence of the TGM for $A_{N}$ has been analyzed previously. For the case $d(x,y,t)\geq 0$, although $A_{N}^{(m)}$ may no longer be a BTTB matrix, the coefficient matrix at each time level can be expressed as $A^{(m)}_{N}=A_{N}+\frac{1}{\eta_{1}}D^{(m)}$, and it follows from Lemma \ref{yl4.4} that

\begin{cor}\rm \label{tl1} Let $A^{(m)}_{N}$ be given by  \eqref{eq5}
	. If the coefficient  function $d(x,y,t)$ is nonnegative and uniformly bounded, then there exists a constant $\alpha$ independent of the  time $m$ such that the weighted Jacobi smoother $G$ satisfies \eqref{mgsz}.
	
	\end{cor}
\begin{proof}
Since $\frac{1}{\eta_{1}} D^{(m)}$ is nonnegative for any $1\leq m\leq M$. By Lemma \eqref{yl4.4}, we have
	\begin{align*}
\rho (\mathrm{diag}(A^{(m)}_{N})^{-1}A^{(m)}_{N})\leq \rho (\mathrm{diag}(A_{N})^{-1}A_{N}).
	\end{align*}
	Moreover, 
	\begin{align*}
	0<\frac{2a_{0}}{\|f_{\gamma_{1},\gamma_{2}}\|_{\infty}}\leq \frac{2}{\rho(\mathrm{diag}(A_{N})^{-1}A_{N})}\leq \frac{2}{\rho(\mathrm{diag}(A^{(m)}_{N})^{-1}A^{(m)}_{N})},
	\end{align*}
	and \(\frac{2a_{0}}{\|f_{\gamma_{1},\gamma_{2}}\|}\) is independent of the  time level \(m\), there always exists an $\alpha$  independent of \(m\) such that \eqref{mgsz} holds.
	 The Proof is completed.
\end{proof}

\begin{cor}\rm Under the conditions of Corollary \ref{tl1}, there exists a constant $\beta$ independent of the time $m$ such that the prolongation operator defined by \eqref{tssz} satisfies \eqref{ytsz}.	
\end{cor}
\begin{proof}
	By Theorem 3.2, there exists a positive constant $\beta_0$,  such that
	\begin{align*}
		\min_{e^{H}\in \mathbb{R}^{n_{c}}} \|e^{h} - I^{h}_{H} e^{H}\|^{2}_{0,A_{N}} \leq \beta_{0} \|e^{h}\|^{2}_{1,A_{N}}, \quad \forall e^{h}\in \mathbb{R}^{n}.
	\end{align*}
	By the uniform boundedness of $d(x,y,t)$, there exists a positive constant $\hat d$ such that
	\[
	d(x,y,t) \le \hat d \quad \text{for all } (x,y,t)\in \Omega\times(0,T).
	\]
It follows from Lemma \ref{xyl2} that
\begin{align*}
		\min_{e^{H}\in \mathbb{R}^{n_{c}}} \|e^{h} - I^{h}_{H} e^{H}\|^{2}_{0,A^{(m)}_{N}} \leq \beta_{0}\left(1+\frac{\hat d}{\eta_{1} a_{0}}\right) \|e^{h}\|^{2}_{1,A^{(m)}_{N}}, \quad \forall e^{h}\in \mathbb{R}^{n}.
\end{align*}	
Let
$\beta = \beta_0 \left(1 + \frac{\hat d}{\eta_{1} a_0}\right)$.
	Then $\beta$ does not depend on  the time instant $m$.

\noindent	The proof is completed.
\end{proof}

\begin{thm}\rm \label{TH3.3}
	Let $A^{(m)}_{N}$ be given by \eqref{eq5}. If the coefficient function $d(x,y,t)$ is nonnegative and uniformly bounded, $G$ is the weighted Jacobi method with $\omega \in (0,\frac{2a_{0}}{\|f_{\gamma_{1},\gamma_{2}}\|_{\infty}})$, and $I^{h}_{H}$ is as in (\ref{tssz}), then the TGM method converges.
\end{thm}

\begin{cor}\rm
Under the conditions of Theorem \ref{TH3.3}, $\omega\in(0,1]$ serves as a sufficient condition for the convergence of the TGM with weighted Jacobi as the smoother when solving \eqref{eq5}.
\end{cor}
\begin{proof}
	We only need to prove that
	\begin{align*}
	\frac{2a_0}{\left\|f_{\gamma_1,\gamma_2}\right\|_\infty}>1
	\end{align*}
	for any \((\gamma_1,\gamma_2)\in(1,2)\). A direct calculation gives
	\begin{align*}
		\frac{2a_0}{\left\|f_{\gamma_1,\gamma_2}\right\|_\infty}
		&= \frac{-4g_1^{(\gamma_1)} - 4\frac{\eta_2}{\eta_1}g_1^{(\gamma_2)}}
		{\left\| f_{\gamma_1}(\theta_1) + \frac{\eta_2}{\eta_1}f_{\gamma_2}(\theta_2) \right\|_\infty} = \frac{4\gamma_1 + 4\frac{\eta_2}{\eta_1}\gamma_2}
		{2^{\gamma_1+1}
		+	 \frac{\eta_2}{\eta_1}2^{\gamma_2+1}} > 1.
	\end{align*}
	The Proof is completed.
\end{proof}

The above corollary indicates that regardless of how the conditions of equation \eqref{new1} vary, as long as the relaxation parameter  \(\omega \in (0,1]\), the TGM is always convergent, which provides a conservative but safe interval for parameter selection. Furthermore, the surface of \(\frac{2a_{0}}{\|f_{\gamma_{1},\gamma_{2}}\|_{\infty}}\) corresponding to different values of \((\gamma_{1},\gamma_{2})\) when \(N=64\) is shown in Figure \ref{energy}. It can be observed that the surface representing the values of \(\frac{2a_{0}}{\|f_{\gamma_{1},\gamma_{2}}\|_{\infty}}\) lies strictly above the light blue plane, and \(\frac{2a_{0}}{\|f_{\gamma_{1},\gamma_{2}}\|_{\infty}}\) achieves an approximate maximum value of \(1.0615\) when \((\gamma_{1},\gamma_{2})=(1.445,1.445)\). In the following, we use an example to show that the choice of \(\omega\) will significantly affect the iteration efficiency of the TGM.

\begin{figure}[h]
	\begin{minipage}{1.0 \linewidth}
		\centering
		\includegraphics[width=4in]{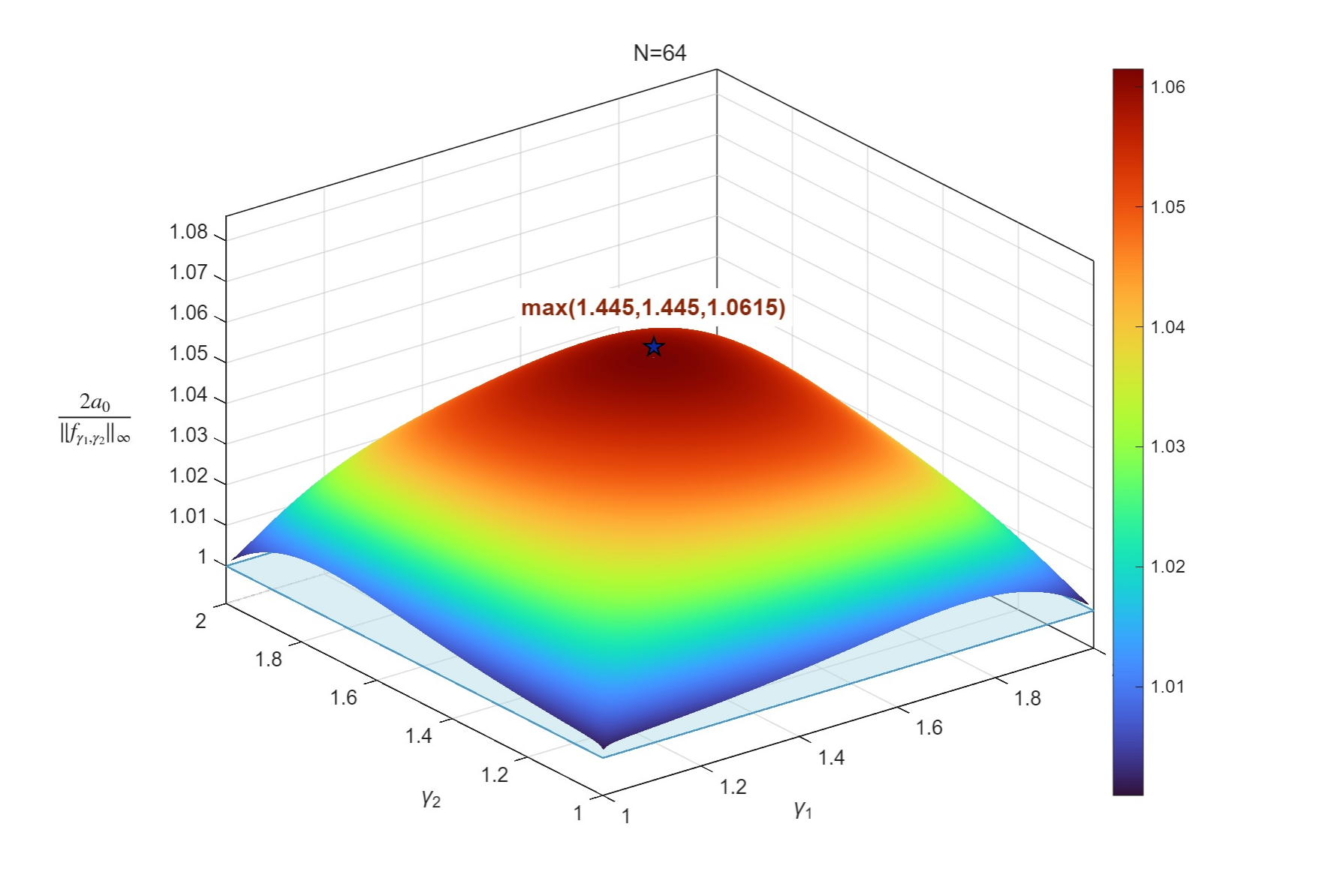}
	\end{minipage}
	\caption{
	The range of  $\frac{2a_{0}}{\|f_{\gamma_{1},\gamma_{2}}\|_{\infty}}$ when  $N=64$.
		 }\label{energy}
\end{figure}

	\begin{example}\rm
	Consider a two-dimensional Riesz space fractional diffusion equation with \(T=1\) and \(\Omega=(0,1)^2\). The exact solution is taken as
	\[
	u=e^{-t}(x-x^2)^2(y-y^2)^2,
	\]
	and the source term \(f\) is derived from the exact solution. Take \(\gamma_1=1.1\), \(\gamma_2=1.7\), \(N=64\), \(M=4\), and consider the linear system at the first time level. We apply the TGM to solve the linear system obtained from discretization, terminate the iteration when
	\[
	\frac{\|Au-b\|_\infty}{\|b\|_\infty}<10^{-7},
	\]
	and count the number of cycles. Specifically,
	
	(1): The \(d(x,y,t)=0\): We set \(\omega=0.0001\) as the initial value and traverse to \(\omega=1.1\) with a step size of \(0.0001\). Denoting one TGM iteration as one cycle, we record the cycle counts required to reach the stopping tolerance for different values of \(\omega\), and plot the corresponding iteration curve, namely Fig.~\ref{fig:1a}. A direct calculation gives
	\[
	\frac{2a_0}{\|f_{\gamma_1,\gamma_2}\|_\infty}\approx 1.0417,
	\]
	which is marked by a red line in Fig.~\ref{fig:1a}.
	
	\begin{figure}[h]
		\centering
		\begin{subfigure}[b]{0.5\linewidth}
			\centering
			\includegraphics[width=\linewidth]{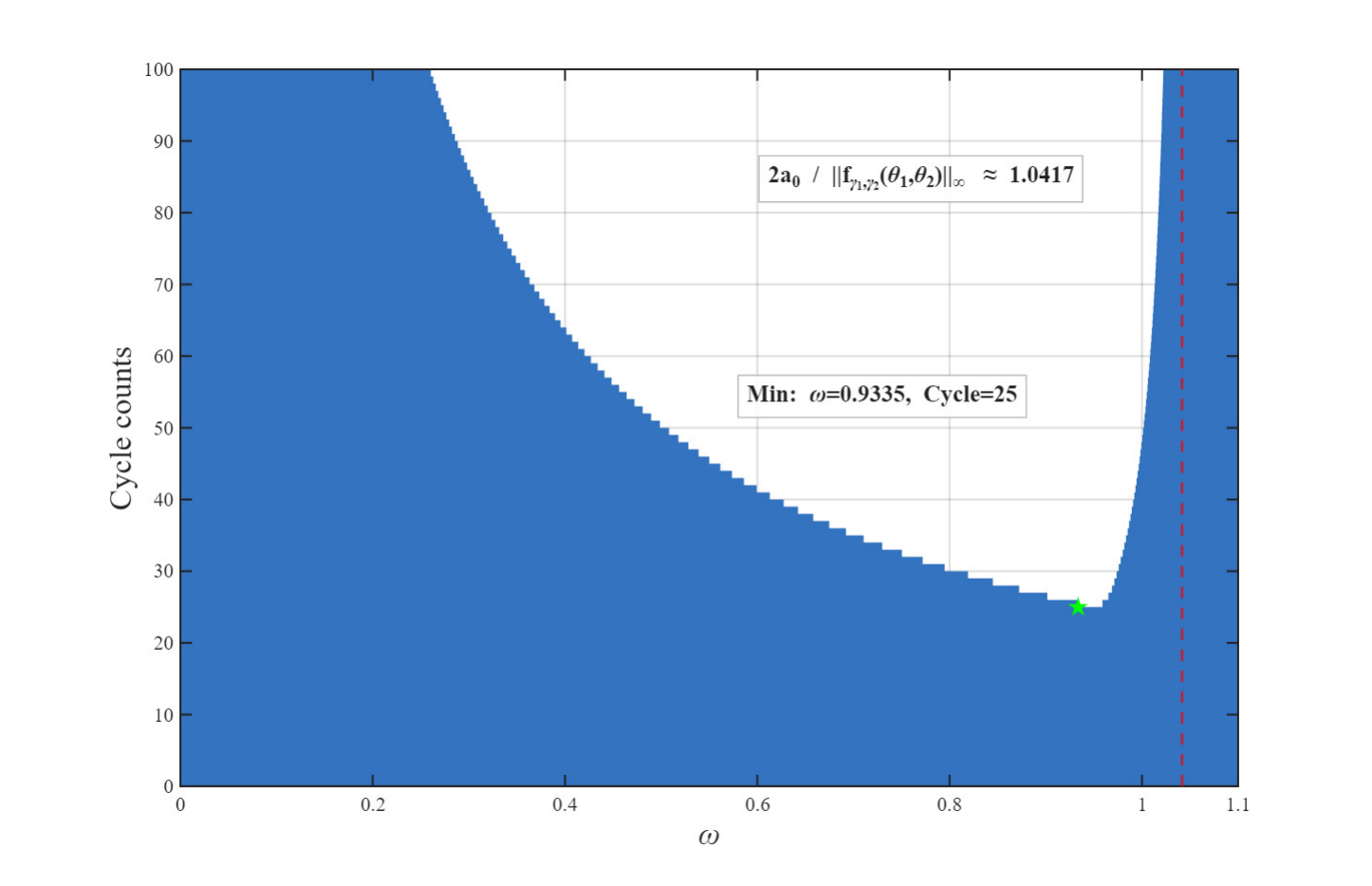}
			\caption{The relationship between \(\omega\) and the cycle counts}
			\label{fig:1a}
		\end{subfigure}
		\hfill
		\begin{subfigure}[b]{0.49\linewidth}
			\centering
			\includegraphics[width=\linewidth]{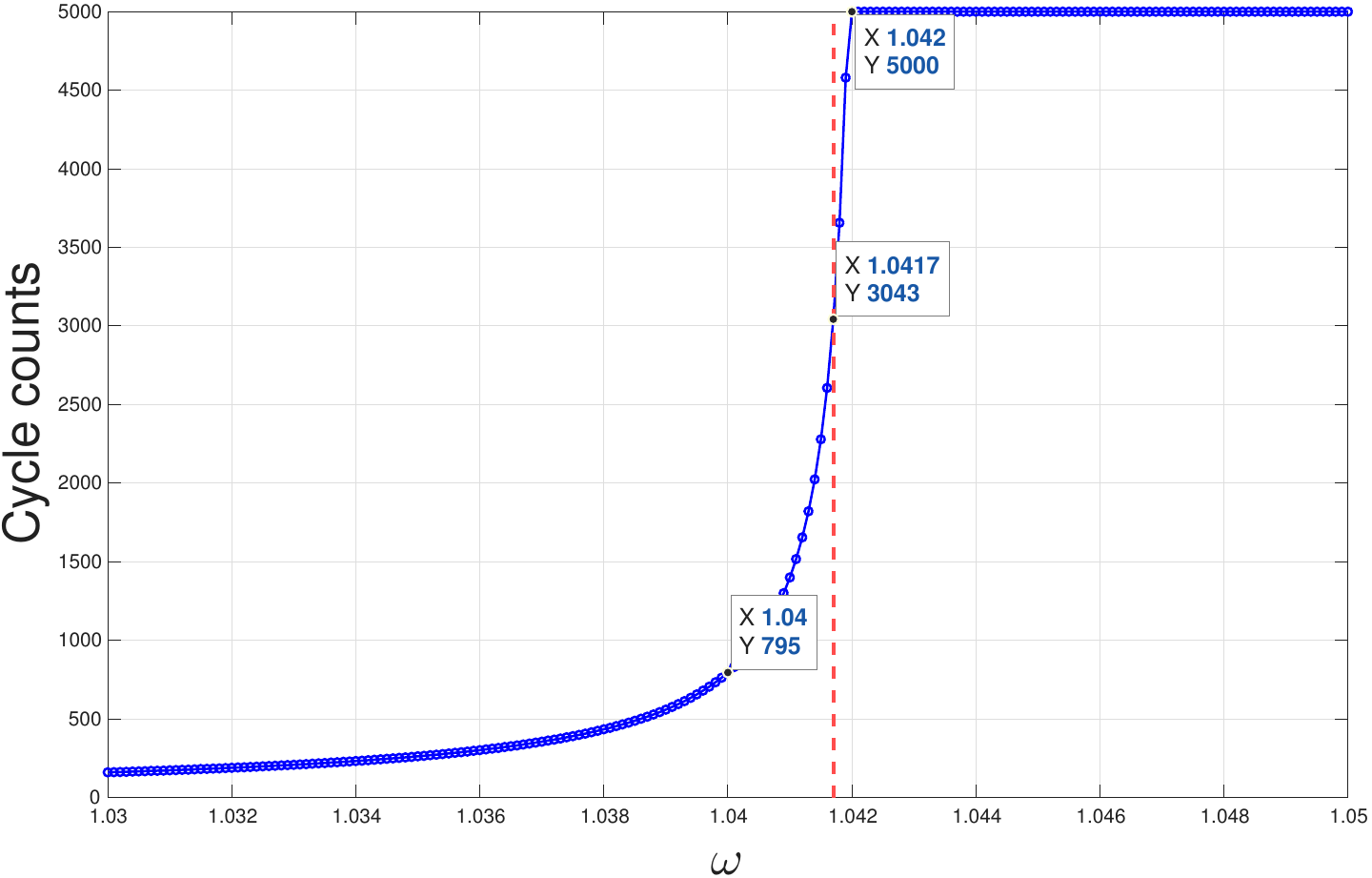}
			\caption{The iterative behavior near \(\omega=1.0417\)}
			\label{fig:1b}
		\end{subfigure}
		\caption{The iteration curves of the TGM for \(d(x,y,t)=0\).}
		\label{fig:1}
	\end{figure}
	
	As can be seen from Fig.~\ref{fig:1a}, the choice of the relaxation parameter \(\omega\) significantly affects the iteration efficiency of the TGM. For example, when \(\omega=0.9335\), the required number of cycles is the smallest, only 25; whereas in other regions, the corresponding cycle counts increase markedly, even exceeding 100. In addition, all values of \(\omega\) to the left of the red line ensure the convergence of the TGM, which verifies the effectiveness of Theorem~\ref{zyyl}.
	
	Fig.~\ref{fig:1b} shows the steep growth phenomenon near \(\omega=1.0417\): when \(\omega=1.04\), the number of cycles required to reach the convergence tolerance is 795; when \(\omega=1.0417\), the required number of cycles increases sharply to 3043; and when \(\omega=1.042\), the TGM still fails to converge after more than 5000 iterations. These phenomena indicate that the choice of \(\omega\) has a great influence on both the convergence and the convergence speed of the TGM, so selecting an appropriate \(\omega\) is necessary for accelerating convergence.
	
	(2): \(d(x,y,t)=(1+x)^7(1+y)^7\):
	\begin{figure}[h]
		\centering
		\begin{subfigure}[b]{0.5\linewidth}
			\centering
			\includegraphics[width=\linewidth]{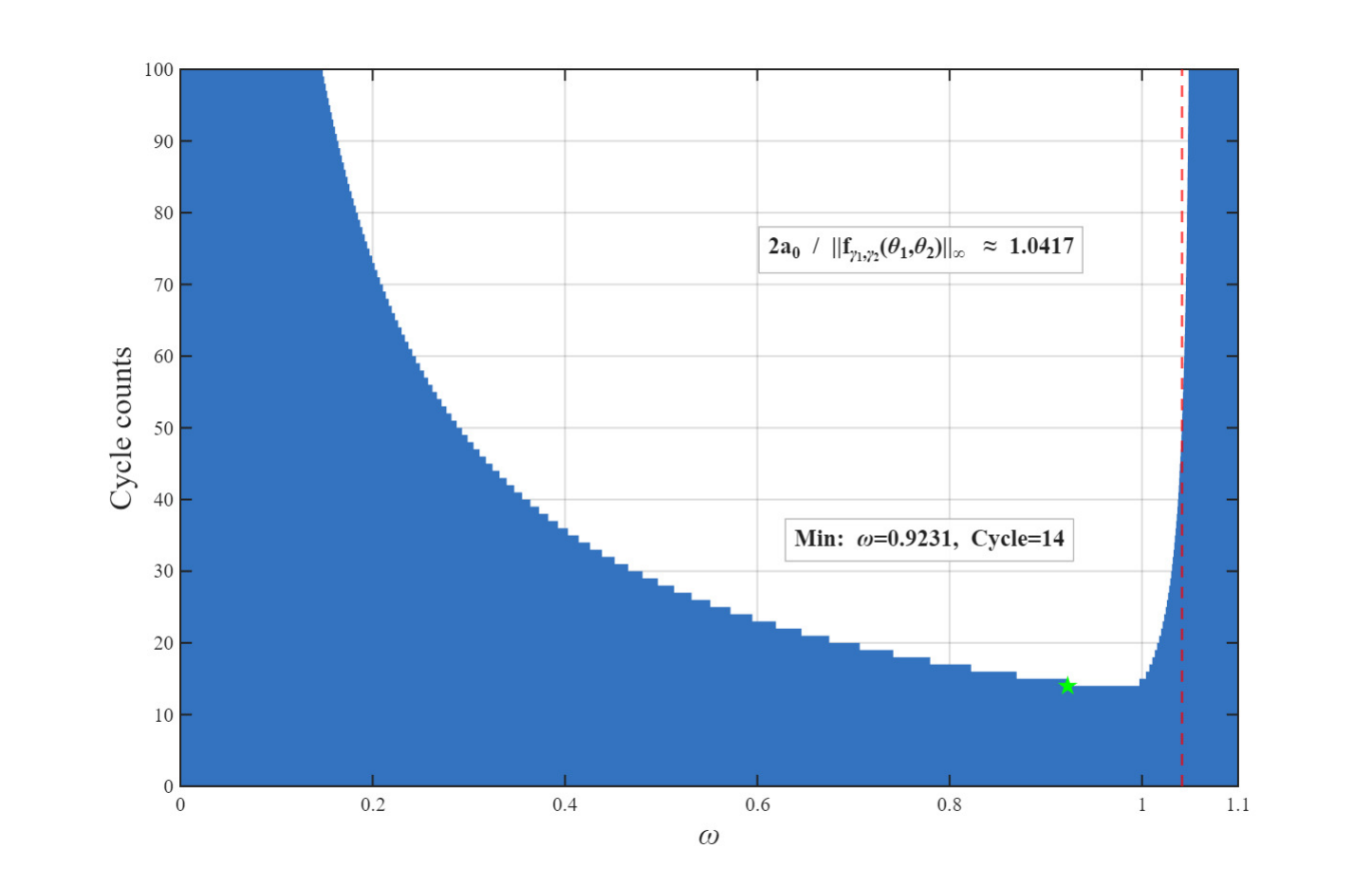}
			\caption{The relationship between \(\omega\) and the cycle counts}
			\label{fig:2a}
		\end{subfigure}
		\hfill
		\begin{subfigure}[b]{0.49\linewidth}
			\centering
			\includegraphics[width=\linewidth]{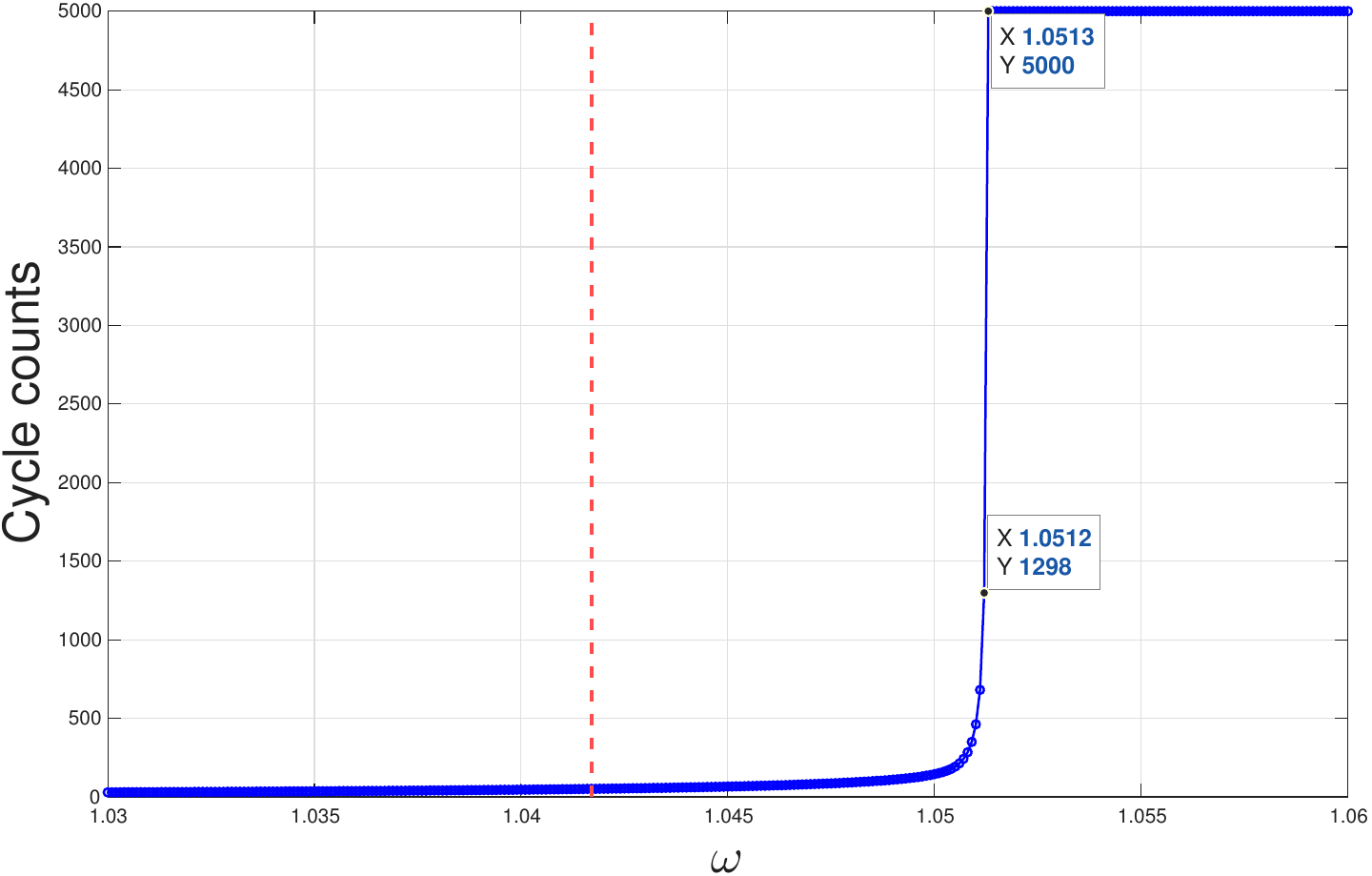}
			\caption{The iterative behavior near \(\omega=1.0417\)}
			\label{fig:2b}
		\end{subfigure}
		\caption{The iteration curves of the TGM for \(d(x,y,t)=(1+x)^7(1+y)^7\).}
		\label{fig:2}
	\end{figure}	
	Similar to the results obtained in (1), Fig.~\ref{fig:2a} shows a sweep of the number of cycles required to reach convergence accuracy under different values of \(\omega\). It can be seen that the number of iterations is minimized at \(\omega=0.9231\), requiring only 14 cycles.	In Fig.~\ref{fig:1b}, the parameters to the right of \(\omega=1.0417\) basically make the number of iterations required by TGM exceed 5000, whereas in Fig.~\ref{fig:2b}, the parameters to the left of \(\omega=1.0513\) all allow TGM to converge. It can be seen that the boundary parameter that ensures the smoothing property is shifted noticeably to the right, which verifies the validity of Corollary~\ref{tl1}. In addition, when \(\omega=1.0512\), the required number of cycles is 1298, while when \(\omega=1.0513\), the iteration still fails to converge after more than 5000 steps. This steep growth phenomenon is far more pronounced than in the case of \(d(x,y,t)=0\), reflecting that the parameter variation is more complicated when \(d(x,y,t)\geq 0\).

\end{example}

\section{Data-driven acceleration strategy}\label{GPR}
As the discretization number \(N\) increases, the discretization of the partial differential equation yields a family of matrices. When solving with parameterized iterative methods, there exists an optimal parameter for each problem size, which can be obtained via fine step grid search and corresponds to the minimal computational cost. As \(N\) continues to grow, properties such as the matrix spectrum change, and these optimal parameters usually fluctuate. Such fluctuations generally exhibit nonlinear characteristics and cannot be characterized by rigorous mathematical expressions, determining them via exact eigenvalue computations incurs a high computational cost\cite{STT}.

Gaussian process regression (GPR) provides an effective data-driven strategy for learning such parameters: it uses kernel functions to model the varying relationship between problem sizes and quasi-optimal parameters, and predicts near-optimal parameters at a low computational cost. Assume that we have a training set containing $d$ data:
\[
\mathcal{D}=\left\{\left(N_{i},\omega_{i}\right)\,\middle|\,i=1,2,\ldots,d\right\}=:\left\{\boldsymbol{N},\boldsymbol{\omega}\right\},
\]
where the pair $(N_{i},\omega_{i})$ denotes the $N_i$-th mesh partition number and the corresponding optimal relaxation parameter. The goal of Gaussian process regression is to establish a mapping between the
input set $\boldsymbol{N}$ and the output set $\boldsymbol{\omega}$, represented as
$\mathcal{F}(N):\mathbb{R}\rightarrow\mathbb{R}$, and to forecast the potential output value
$\tilde{\omega}_{*}=\mathcal{F}(N_{*})$ at a new test point $N_{*}$.
In practical linear regression scenarios, we model the observed value \(\omega\) as the result of a latent function $\mathcal{F}(N)$ perturbed by additive noise $\xi$, i.e.,
\begin{equation}
\omega = \mathcal{F}(N) + \xi.
\end{equation}
Suppose  that $\xi$ follows a Gaussian distribution with a zero mean and variance
$\sigma^{2}$, that is $\xi\sim\mathcal{N}\big(0,\sigma^{2}\big)$. The common
range for $\sigma$ is $[10^{-6},10^{-2}]$. For this investigation, we have set
$\sigma=10^{-4}$. 
 Without loss of generality, assume true function $\mathcal{F}(N)$ follows a zero mean Gaussian process:
\begin{equation*}
	\begin{bmatrix}
		\mathcal{F}(N_{1})\\
		\vdots\\
		\mathcal{F}(N_{d})
	\end{bmatrix}
	\sim
	\mathcal{N}\left(
	\begin{bmatrix}
		0\\
		\vdots\\
		0
	\end{bmatrix},
	\begin{bmatrix}
		\kappa\big(N_{1},N_{1}\big) & \cdots & \kappa\big(N_{1},N_{d}\big)\\
		\vdots & \ddots & \vdots\\
		\kappa\big(N_{d},N_{1}\big) & \cdots & \kappa\big(N_{d},N_{d}\big)
	\end{bmatrix}
	\right),
\end{equation*}
where	$\kappa(N_{i},N_{j})$ defines the covariance between the function values corresponding to the input points $N_{i}$ and $N_{j}$.
The kernel function plays a pivotal role in GPR, as it is responsible for generating the covariance matrix that quantifies the similarity between any two input variables. The closer the distance between inputs, the higher the correlation between their corresponding output variables. Hence, it is essential to select or construct kernel functions that meet the specific demands of the problem at hand. In this paper, we consider the Gaussian kernel function, which is given by
\begin{align*}
\kappa(N_{i}, N_{j}) = \delta_{\mathcal{F}}^{2} \exp\left(-\frac{\left\|N_{i}-N_{j}\right\|_{2}^{2}}{2\delta_{\mathcal{L}}^{2}}\right),
\end{align*}
 where $\delta_\mathcal{F}$, $\delta_{\mathcal{L}}$ are hyperparameters,and the specific calculation of the hyperparameters is given in \cite{Su}. Define $	K\big(\boldsymbol{N},\boldsymbol{N}\big)=\left(\kappa(N_{i},N_{j})\right)\in \mathbb R^{d\times d}$.
 Then, we can express the prior distribution of the observed values $\boldsymbol{\omega}$ as
\begin{equation}
	\boldsymbol{\omega}\sim
	\mathcal{N}\left(\boldsymbol{0},\,
	K\big(\boldsymbol{N},\boldsymbol{N}\big)+\delta^{2}I_{d}\right).
\end{equation}
The joint prior distribution of the observed value $\boldsymbol{\omega}$ and the
predicted value $\boldsymbol{\omega}_{*}$ is given by
\begin{equation}
	\begin{bmatrix}
		\boldsymbol{\omega}\\
		\boldsymbol{\omega}_{*}
	\end{bmatrix}
	\sim
	\mathcal{N}\left(
	\begin{bmatrix}
		\boldsymbol{0}\\
		\boldsymbol{0}
	\end{bmatrix},
	\begin{bmatrix}
		K\big(\boldsymbol{N},\boldsymbol{N}\big)+\delta^{2}I_{d} & K\big(\boldsymbol{N},\boldsymbol{N}_{*}\big)\\
		K\big(\boldsymbol{N}_{*},\boldsymbol{N}\big) & K\big(\boldsymbol{N}_{*},\boldsymbol{N}_{*}\big)
	\end{bmatrix}
	\right).
\end{equation}
Then the conditional expectation and variance of $\mathcal{F} (n)$ is given by
\begin{equation}
	\mathbb{E}[\mathcal{F}(N)|\boldsymbol{\omega}]
	=K\big(\boldsymbol{N}_{*},\boldsymbol{N}\big)
	\left[K\big(\boldsymbol{N},\boldsymbol{N}\big)+\delta^{2}I_{d}\right]^{-1}
	\boldsymbol{\omega},
\end{equation}
\begin{equation}
Var[\mathcal{F}(N)|\boldsymbol{\omega}]
	=K\big(\boldsymbol{N}_{*},\boldsymbol{N}_{*}\big)
	-K\big(\boldsymbol{N}_{*},\boldsymbol{N}\big)
	\left[K\big(\boldsymbol{N},\boldsymbol{N}\big)+\delta^{2}I_{d}\right]^{-1}
	K\big(\boldsymbol{N},\boldsymbol{N}_{*}\big).
\end{equation}
For the output $\omega_{*}$ in the test set, we use the mean of the Gaussian
regression as its estimate, i.e. $\widehat{\omega}_{*}=\mathbb{E}(\mathcal{F}(N)|\boldsymbol{\omega})$.

\section{Numerical experiments}\label{Sec5}

In Section \ref{TGM}, we discuss the convergence of the TGM. The TGM contains only one coarse level. When the size of the original matrix is too large, the order of the coarse matrix obtained by projection may still be high, and it remains difficult to solve the coarse-grid equation exactly. A natural approach to resolve this difficulty is to recursively apply the TGM on this coarse level to generate coarser levels successively, until the order of the coarsest level is sufficiently small. Except for the coarsest level, each of the remaining grid levels performs only one such recursive procedure, which yields the V-cycle of the multigrid method\cite{Ruge,Stu,XJC}.

\begin{figure}[h]
	\begin{minipage}{1.0 \linewidth}
		\centering
		\includegraphics[width=4in]{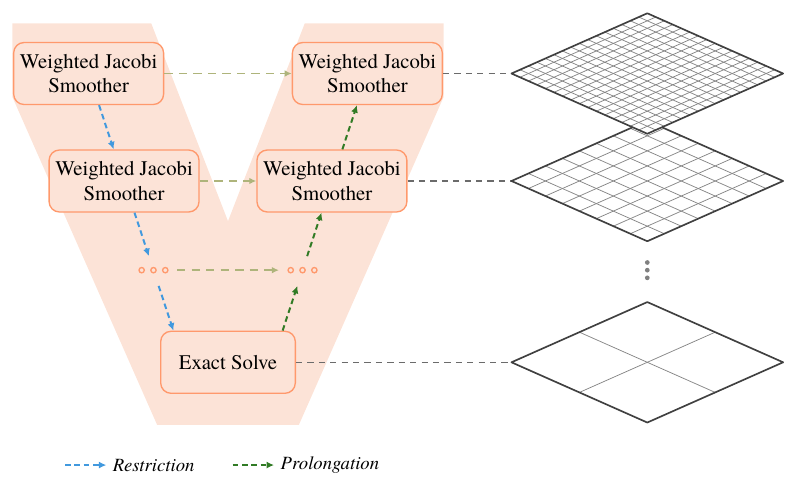}
	\end{minipage}
	\caption{ V-cycle sketch map   }
\end{figure}

In this section, we present several numerical experiments to demonstrate the computational efficiency of the multigrid method and verify the acceleration effect of using GPR-predicted parameters on the multigrid solver. Specifically, we set the order of the matrix on the coarsest level of the multigrid method to be no more than 1000. Except for the coarsest level, one weighted Jacobi iteration is applied for pre-smoothing and post-smoothing on all remaining grid levels, and the same input relaxation parameter is used for all weighted Jacobi iterations. In the experiments, we consider three test cases for comparison: the multigrid method with \(\omega=1\) as the relaxation parameter, the multigrid method using GPR-predicted values as the relaxation parameter, and the unpreconditioned conjugate gradient method.

The training strategy of GPR is as follows: starting from \(N=5\) with a step size of 5 up to \(N=50\), we construct a training set consisting of 10 samples. For each problem size, the quasi-optimal parameter \(\omega\) is obtained via grid search over the interval \([0.001,1]\) with a step size of 0.001.For all algorithms, the zero vector is used as the initial iterate, with the stopping criterion
\begin{align*}
	\text{RRES}=\frac{\|Au-b\|_{\infty}}{\|b\|_{\infty}}\leq 10^{-7}.
\end{align*}
All experiments are carried out using MATLAB (version R2025b) on a laptop equipped with an Intel Core i5-9300H CPU (4 cores, 2.40 GHz) and 8 GB of memory.

\begin{example}\rm  Consider the two-dimensional Riesz space fractional diffusion equation (\ref{new1}) defined on the domain $(0, 1) \times (0, 1) \times (0, 1)$. The exact solution is given by $u(x, y, t) = 2(t^3 + 0.01)xy(1 - x)(1 - y)$, and the initial condition is $u_0(x, y) = 0.02xy(1 - x)(1 - y)$. We take $d(x,y,t)=\ln(1+x)\,(y+1)^2 $.  The source term $f$ is derived from the exact solution.

\begin{table}[H]
	\centering
	\caption{Example 5.1}
	\label{table2}
	\small
	\setlength{\tabcolsep}{7pt}
	\renewcommand{\arraystretch}{1.12}
	\begin{tabular}{c l l c c c c}
		\toprule
		\multirow{2}{*}{$(\gamma_{1},\gamma_{2})$} 
		& \multirow{2}{*}{Method} 
		& \multirow{2}{*}{Index}
		& \multicolumn{4}{c}{$N$} \\
		\cmidrule(lr){4-7}
		& & & 64 & 128 & 256 & 512 \\
		\midrule
		
		\multirow{6}{*}{$(1.4,1.4)$}
		& \multirow{2}{*}{Default} & IT  & 27  & 21  & 19  & 18  \\
		&                          & CPU & 0.1710 & 1.0781 & 3.7969 & 20.7812 \\
		\cmidrule(lr){2-7}
		& \multirow{2}{*}{GPR}     & IT  & 13  & 14  & 14  & 14  \\
		&                          & CPU & 0.0625 & 0.3596 & 2.1875 & 15.3594 \\
		\cmidrule(lr){2-7}
		& \multirow{2}{*}{CG}      & IT  & 97  & 161 & 278 & 447 \\
		&                          & CPU & 0.2969 & 2.3125 & 32.7656 & 440.1875 \\
		\midrule
		
		\multirow{6}{*}{$(1.5,1.6)$}
		& \multirow{2}{*}{Default} & IT  & 30  & 24  & 20  & 17  \\
		&                          & CPU & 0.2656 & 1.1562 & 3.5781 & 21.0781 \\
		\cmidrule(lr){2-7}
		& \multirow{2}{*}{GPR}     & IT  & 16  & 17  & 18  & 19  \\
		&                          & CPU & 0.1094 & 0.5156 & 2.8212 & 21.6406 \\
		\cmidrule(lr){2-7}
		& \multirow{2}{*}{CG}      & IT  & 111 & 199 & 345 & 585 \\
		&                          & CPU & 0.1875 & 2.5940 & 42.6407 & 552.3438 \\
		\midrule
		
		\multirow{6}{*}{$(1.7,1.9)$}
		& \multirow{2}{*}{Default} & IT  & 72  & 60  & 44  & 33  \\
		&                          & CPU & 0.3906 & 2.0469 & 7.2656 & 39.1406 \\
		\cmidrule(lr){2-7}
		& \multirow{2}{*}{GPR}     & IT  & 23  & 26  & 30  & 31  \\
		&                          & CPU & 0.0938 & 0.8125 & 4.3975 & 37.4217 \\
		\cmidrule(lr){2-7}
		& \multirow{2}{*}{CG}      & IT  & 155 & 303 & 567 & 1072 \\
		&                          & CPU & 0.2812 & 3.7969 & 67.3438 & 1067.7187 \\
		\bottomrule
	\end{tabular}
\end{table}

	It can be seen that the GPR parameter selection strategy is significantly superior to the Default strategy in most cases. For example, in Table~\ref{table2} for $(\gamma_1,\gamma_2)=(1.4,1.4)$ and $N=512$, Default requires 18 iterations and 20.7812~s, whereas GPR requires only 14 iterations and 15.3594~s; in Table~2 for $(\gamma_1,\gamma_2)=(1.8,1.8)$ and $N=512$, Default requires 18 iterations and 20.4062~s, whereas GPR requires only 14 iterations and 15.9375~s. As $N$ increases, the iteration counts of GPR remain basically in the range of 13--31, exhibiting a favorable mesh-independent convergence trend; although the iteration counts of Default decrease with increasing $N$ in some test cases, they are still overall higher than those of GPR. The iteration counts and CPU times of CG grow fastest with $N$; at $N=512$, CG usually requires hundreds of iterations and hundreds of seconds of CPU time, which is far higher than those of the two multigrid methods. For a few parameter combinations, the CPU times of Default and GPR are already relatively close. However, over all test cases, the relaxation parameters predicted by GPR can effectively reduce both the iteration counts and the computing time.

\end{example}

\begin{example}\rm  Consider the two-dimensional Riesz space fractional diffusion equation (\ref{new1}) defined on the domain $(0, 1) \times (0, 1) \times (0, 1)$. The exact solution is given by $u(x, y, t) = e^{-2t}\sin(\pi x)\sin(\pi y)$, and the initial condition is $u_0(x, y) = \sin(\pi x)\sin(\pi y)$.We take $d(x,y,t)=10^{-6}$.  The source term $f$ is derived from the exact solution.

	\begin{table}[H]
		\centering
		\caption{Example 5.2}
		\label{table3}
		\small
		\setlength{\tabcolsep}{7pt}
		\renewcommand{\arraystretch}{1.12}
		\begin{tabular}{c l l c c c c}
			\toprule
			\multirow{2}{*}{$(\gamma_{1},\gamma_{2})$} 
			& \multirow{2}{*}{Method} 
			& \multirow{2}{*}{Index}
			& \multicolumn{4}{c}{$N$} \\
			\cmidrule(lr){4-7}
			& & & 64 & 128 & 256 & 512 \\
			\midrule
			
			\multirow{6}{*}{$(1.2,1.2)$}
			& \multirow{2}{*}{Default} & IT  & 39  & 29  & 22  & 21  \\
			&                          & CPU & 0.1875 & 1.2812 & 4.0469 & 23.8595 \\
			\cmidrule(lr){2-7}
			& \multirow{2}{*}{GPR}     & IT  & 13  & 14  & 14  & 14  \\
			&                          & CPU & 0.0781 & 0.5312 & 2.8906 & 17.0781 \\
			\cmidrule(lr){2-7}
			& \multirow{2}{*}{CG}      & IT  & 71  & 117 & 189 & 302 \\
			&                          & CPU & 0.2031 & 1.3906 & 22.8125 & 397.6406 \\
			\midrule
			
			\multirow{6}{*}{$(1.5,1.5)$}
			& \multirow{2}{*}{Default} & IT  & 30  & 24  & 17  & 16  \\
			&                          & CPU & 0.3281 & 1.0625 & 3.2188 & 18.4712 \\
			\cmidrule(lr){2-7}
			& \multirow{2}{*}{GPR}     & IT  & 13  & 13  & 13  & 14  \\
			&                          & CPU & 0.1406 & 0.3125 & 2.4219 & 16.5000 \\
			\cmidrule(lr){2-7}
			& \multirow{2}{*}{CG}      & IT  & 72  & 125 & 215 & 369 \\
			&                          & CPU & 0.1406 & 1.5469 & 24.6250 & 419.6719 \\
			\midrule
			
			\multirow{6}{*}{$(1.8,1.8)$}
			& \multirow{2}{*}{Default} & IT  & 59  & 49  & 34  & 18  \\
			&                          & CPU & 0.2970 & 1.7500 & 6.0938 & 20.4062 \\
			\cmidrule(lr){2-7}
			& \multirow{2}{*}{GPR}     & IT  & 14  & 14  & 14  & 14  \\
			&                          & CPU & 0.0781 & 0.3906 & 2.7031 & 15.9375 \\
			\cmidrule(lr){2-7}
			& \multirow{2}{*}{CG}      & IT  & 83  & 161 & 302 & 568 \\
			&                          & CPU & 0.2031 & 2.1406 & 36.0000 & 688.1875 \\
			\bottomrule
		\end{tabular}
	\end{table}

It can be seen that the GPR parameter selection strategy is overall significantly superior to Default. For example, for $(\gamma_1,\gamma_2)=(1.2,1.2)$ and $N=512$, Default requires 21 iterations and 23.8595~s, whereas GPR requires only 14 iterations and 17.0781~s; for $(\gamma_1,\gamma_2)=(1.8,1.8)$ and $N=512$, Default requires 18 iterations and 20.4062~s, whereas GPR requires only 14 iterations and 15.9375~s. As $N$ increases, the iteration counts of GPR remain basically stable at 13--14, exhibiting a favorable mesh-independent convergence trend; although the iteration counts of Default also decrease with increasing $N$, they are still overall higher than those of GPR, and the gap is particularly evident at $N=64$ and $N=128$. The iteration counts and CPU times of CG grow fastest with $N$; at $N=512$, CG usually requires hundreds of iterations and hundreds of seconds of CPU time. For instance, for $(\gamma_1,\gamma_2)=(1.8,1.8)$, CG requires 568 iterations and 688.1875~s, which is far higher than those of the two multigrid methods. For a few parameter combinations (such as $(\gamma_1,\gamma_2)=(1.5,1.5)$ and $N=512$), the CPU times of Default and GPR are already relatively close, but GPR is still no worse than Default. Over all test cases, the relaxation parameters predicted by GPR can effectively reduce both the iteration counts and the computing time, further verifying the effectiveness and superiority of the data-driven parameter selection strategy for accelerating the multigrid method.
\end{example}

\section{Conclusion}\label{conclusion} In this paper, the weighted Jacobi method is adopted as the smoother in a multigrid method for solving the two-dimensional Riesz space fractional diffusion equation. To address the problem of relaxation parameter selection, a GPR-based parameter selection strategy is proposed. Numerical experiments verify the correctness of the relevant theory and confirm the high computational efficiency of this data-driven algorithm.
Moreover, the combination of neural networks and traditional numerical methods has become increasingly important in recent years. In fact, neural networks exhibit spectral bias, which makes them naturally complementary to multigrid methods. As a result, a growing number of researchers have combined the two and obtained many important results. Nevertheless, reliable mathematical explanations for these works remain lacking. In our subsequent work, we will provide corresponding analysis and design more efficient and robust algorithms.\\

\noindent \textbf{Data availability:} No data was used for the research described in the article.\\

\noindent \textbf{Acknowledgements:} The authors are thankful to the editors and the anonymous  referees for their valuable comments to improve the paper.\\

\noindent \textbf{Funding:}  This work was supported by the National Natural Science Foundation of China (12671450), the National Key Research and Development Program of China (2023YFB3001604), and the Graduate Research Innovation Project of Xiangtan University (XDCX2026Y276).

\end{document}